\documentclass[a4paper,12pt,twoside]{amsart} 
\usepackage{amsmath,amsthm}
\usepackage{amsfonts}
\usepackage{amssymb}
\usepackage[english]{babel}
\usepackage{xcolor}

\newtheorem{theo}{Theorem}[section]
\newtheorem{lem}[theo]{Lemma}
\newtheorem{prop}[theo]{Proposition}
\newtheorem{cor}[theo]{Corollary}

\newcommand{\TT}{{\mathbb T}}
\newcommand{\e}{{\text{e}}}
\newcommand{\C}{{\mathbb C}}

\newcommand{\N}{{\mathbb N}}

\newcommand{\R}{{\mathbb R}}

\newcommand{\pac}{K_{\mbox{\tiny \rm $p$-ac}}}
\newcommand{\cs}{K_{\mbox{\tiny \rm cs}}}
\newcommand{\ac}{K_{\mbox{\tiny \rm ac}}}

\newcommand{\sk}{K_{\mbox{\tiny \rm sk}}}
\newcommand{\kb}{K_{\mbox{\tiny \rm k}}}

\theoremstyle{remark}

\begin{document}

\baselineskip=16pt

\title{Resolvent conditions and growth of powers of operators}
%\title{Resolvent conditions and ergodic properties of operators}

\author{Guy Cohen}
\address{Department of Electrical 
%and Computer 
Engineering, Ben-Gurion University, Beer-Sheva, Israel}
\email{guycohen@bgu.ac.il}

\author{Christophe Cuny}
\address{UMR CNRS 6205, Laboratoire de Math\'ematiques de Bretagne Atlantique, Universite 
de Brest, Brest, France}
%\address{Department of Mathematics, Universit\'e de Bretagne Occidentale, Brest, France}
\email{christophe.cuny@univ-brest.fr}

\author{Tanja Eisner}
\address{Institute of Mathematics, University of Leipzig, Germany}
\email{tatjana.eisner@math.uni-leipzig.de}

\author{Michael Lin}
\address{Department of Mathematics, Ben-Gurion University, Beer-Sheva, Israel}
\email{lin@math.bgu.ac.il}


\begin{abstract}
Following Berm\'udez et al. \cite{BBMP}, we study the rate of growth of the norms of the 
powers of a linear operator, under various resolvent conditions or 
Ces\`aro boundedness assumptions. 
%We show that $T$ is power-bounded if (and only if) both $T$ and $T^*$ are absolutely  
%Ces\`aro bounded.  
In Hilbert spaces, we prove that if $T$ 
satisfies the Kreiss condition, $\|T^n\|=O(n/\sqrt {\log n})$; if $T$ is absolutely Ces\`aro 
bounded, $\|T^n\|=O(n^{1/2 -\varepsilon})$ for some $\varepsilon >0$ (which depends on $T$);
if $T$ is strongly Kreiss bounded, then $\|T^n\|=O((\log n)^\kappa)$ for some $\kappa >0$.
We show that a Kreiss bounded operator on a reflexive space is Abel ergodic,
and its Ces\`aro means of order $\alpha$ converge strongly when $\alpha >1$.
\end{abstract}
\maketitle
\vspace*{-0.6 cm}

\section{Introduction}

\subsection{Background}

The mean ergodic theorem, proved by Yosida and by Kakutani, asserts the convergence in 
norm of the averages $\frac1n \sum_{k=1}^n T^k x$ of a weakly almost periodic 
operator $T$ on a Banach space $X$.  When $T$ is power-bounded, the convergence 
$\frac1n\sum_{k=1}^n T^k x \to y$ is equivalent to the Abel convergence 
$\lim_{r \to 1^-} (1-r)\sum_{n=0}^\infty r^n T^nx =y$.

An example of Hille \cite{Hi} (in $L^1$) shows that power-boundedness is not necessary 
for mean ergodicity. Mean ergodicity easily implies that $\|T^n\|=O(n)$.
Derriennic \cite{De} gave an example of $T$ mean ergodic in a Hilbert space with 
$T^*$ not mean ergodic (so $T$ is not power-bounded, and $\limsup n^{-1}\|T^n\| >0$); 
see also \cite[Example 3.1]{TZ}.
A mean ergodic $T$ in $L^1$ with $\limsup n^{-1}\|T^n\| >0$ was constructed by Kosek
\cite{Ko}.
\smallskip

The purpose of this paper is to study the connections between different resolvent conditions
and Ces\`aro boundedness conditions, and the growth properties of $\|T^n\|$. Our work
continues and complements that of Berm\'udez et al. \cite{BBMP}. For an overview of the 
results see Subsection \ref{overview} below.

\subsection{The Kreiss resolvent condition}

Kreiss \cite{Kr} presented the following resolvent condition ({\it Kreiss resolvent condition})
\begin{equation} \label{KRC}
\|R(\lambda,T)\| \le \frac C{|\lambda| -1} \qquad |\lambda| > 1 \ .
\end{equation}
We shall denote by $\kb=\kb(T)$ the smallest constant $C>0$ for which \eqref{KRC} holds.
Kreiss proved that in finite-dimensional spaces (\ref{KRC}) implies power-boundedness. 
Lubich and Nevanlinna \cite{LN} proved that (\ref{KRC}) implies $\|T^n\|=O(n)$;
this is the best estimate \cite{Sh}, \cite[Theorem 6]{Ne}. 
Earlier, Kreiss gave a resolvent condition  for the generator of a $C_0$-semigroup, 
inspired by the Hille-Yosida theorem, which in finite-dimensional spaces yields boundedness 
of the semigroup; however, in contrast to \cite{LN}, Eisner and Zwart \cite{EZ} constructed 
a $C_0$-semigroup with {\it exponential} growth whose generator satisfies Kreiss's condition.
%Note that for $C_0$-semigroups with  infintesimal generator $A$ satisfying an appropriate  
%Kreiss resolvent condition, the growth of $\e^{tA}$ may be exponential \cite{EZ}.
\smallskip

McCarthy \cite{McC} gave an example of $T$ invertible on $\ell^2(\mathbb Z)$ which 
satisfies the stronger condition ({\it strong Kreiss resolvent condition}, sometimes called
{\it iterated Kreiss condition}):
\begin{equation} \label{SKR}
\|R^k(\lambda,T)\| \le \frac C{(|\lambda| -1)^k} \qquad |\lambda| > 1, \quad k=1,2, \dots
\end{equation}
but is not power-bounded; in the example also $T^{-1}$ satisfies (\ref{SKR}).
Condition (\ref{SKR}) implies that $\|T^n\|=O(\sqrt n)$ \cite{McC}, \cite{LN}. This
estimate is the best possible in general Banach spaces \cite[p. 298]{LN}.
Lyubich \cite{Ly2} obtained a family of examples in $L^p[0,1]$ satisfying (\ref{KRC}) 
but not (\ref{SKR}). 
%Using the Hille-Yosida theorem,
Nevanlinna  \cite[Theorem 2]{Ne} (see also \cite[Proposition 1.1]{Ne2}) proved that 
$T$ satisfies 
 (\ref{SKR}) if and only if for some $M$ we have
\begin{equation} \label{SKR2}
\|\e^{zT}\| \le M \e^{|z|} \qquad \forall z \in \mathbb C.
\end{equation}
 We shall denote by $\sk$ the smallest constant $C>0$ such that \eqref{SKR} holds.
\smallskip

A. Montes-Rodr\'\i guez et al. \cite{MSZ} defined the 
{\it uniform Kreiss resolvent condition}  by
%which strengthens (\ref{KRC}), 
\begin{equation} \label{UKR}
\sup_{n\ge 1} \|\sum_{k=0}^n \frac{T^k}{\lambda^{k+1}}\| \le \frac C{|\lambda|-1} \qquad
|\lambda| > 1. 
%\quad n=1,2, \dots
\end{equation}
%was defined by A. Montes-Rodr\'\i guez et al. \cite{MSZ}.
They showed that (\ref{UKR}) does not imply (\ref{SKR}), and proved that 
(\ref{UKR}) holds if and only if there exists $C>0$ such that
\begin{equation} \label{rotated}
\sup_n \| \frac1n \sum_{k=1}^n (\lambda T)^k \| \le C \qquad \forall |\lambda|=1.
\end{equation}
The proof that (\ref{UKR}) implies (\ref{KRC}) is immediate. 
Gomilko and Zem\'anek \cite{GZ} proved that (\ref{SKR}) implies (\ref{UKR}), hence
(\ref{rotated}); thus in reflexive spaces (\ref{SKR}) implies mean ergodicity, since 
$\|T^n\|=O(\sqrt n)$.  If $T$ is power-bounded, then (\ref{SKR}) 
holds (in an equivalent norm $T$ is a contraction and $C=1$ in (\ref{KRC})). 
By Strikwerda and Wade \cite[p. 352]{SW},  (\ref{KRC}) does not imply (\ref{rotated}). 
Berm\'udez et al. \cite{BBMP} proved that if $T$ on a Hilbert space satisfies
(\ref{UKR}), then $\|T^n\|=o(n)$, and then $T$ is mean ergodic.
%If $T$ on $H$ is invertible with $\|T^{-1}\| \le 1$, and satisfies (\ref{KRC}) with $C=1$,
%then it is unitary \cite{Don}.
%Note that (\ref{rotated}) implies $\|T^n\|=O(n)$.
In Section 5 we prove that a positive Ces\`aro bounded operator on a complex Banach lattice
is uniformly Kreiss bounded.

Van Casteren \cite{VC1} proved that if $T$ is power-bounded invertible  on $H$ with 
$\sigma(T) \subset \mathbb T$, and $T^{-1}$ satisfies (\ref{KRC}) (which is equivalent 
to condition (ii) in van Casteren's theorem), then also $T^{-1}$ is power-bounded (see also
\cite{Na}).  This extended results of \cite{Don}, \cite{GK}, \cite{St}. 
%Gohberg-Krein \cite{GK} and of Stampfli \cite{St}.
\smallskip

Following \cite{BBMP}, we may refer to $T$ which satisfies the (strong, uniform) Kreiss 
resolvent condition  as {\it (strongly, uniformly) Kreiss bounded} (abbreviated as SKB or UKB
respectively).
\smallskip

\subsection{Ces\`aro boundedness conditions}

The mean ergodic theorem implies that $T$ is {\it Ces\`aro bounded}, i.e. 
$\sup_n \|\frac1n \sum_{k=1}^n T^k\| < \infty$. However, for mean ergodicity we require 
also that $T^n/n$ converge strongly to zero. By an old (two-dimensional) example of 
Assani \cite{As} (first presented in \cite{Em0}), there is $T_0$ Ces\`aro bounded for which 
$T_0^n/n$ does not converge to zero even weakly. 
Using this example, we construct $T$ on $\ell^2(\mathbb N)$
which is the identity on the space generated by $\{e_3,\dots,e_k,\dots\}$ and $T_0$ on
the span of $e_1,e_2$. Then $T$ is Ces\`aro bounded not power-bounded.
Since $T_0$ is not power-bounded, by the Kreiss matrix theorem it does not satisfy the 
Kreiss resolvent condition, hence neither does $T$.

Hou and Luo \cite{HL} introduced the notion of 
{\it absolute Ces\`aro boundedness} (ACB): there exists $C>0$ such that 
\begin{equation} \label{ACB}
\sup_n\frac1n \sum_{k=0}^{n-1} \|T^k x\| \le C\|x\| \qquad \forall \  x\in X.
\end{equation}
We shall denote by $\ac$ the smallest constant for which \eqref{ACB} holds.
Berm\'udez et al. \cite{BBMP} proved that (\ref{ACB}) implies $\|T^n\|/n \to 0$; hence in
reflexive spaces ACB implies mean ergodicity.
Absolute Ces\`aro boundedness implies uniform Kreiss boundedness, by the characterization 
(\ref{rotated}).  Berm\'udez et al. \cite{BBMP} constructed a Hilbert space (mean ergodic)
operator satisfying (\ref{UKR}) which is not absolutely Ces\`aro bounded.

Van Casteren \cite{VC2},\cite{VC3} introduced the following condition: $T$ is called 
{\it Ces\`aro square bounded} if there exists $C>0$ such that
\begin{equation} \label{CSB}
\sup_n\frac1n \sum_{k=0}^{n-1} \|T^k x\|^2 \le C\|x\|^2 \qquad \forall \  x\in X.
\end{equation}

Van Casteren \cite{VC2} proved that if {\it both} $T$ and $T^*$ are Ces\`aro square 
bounded in $H$, then $T$ is power-bounded, and gave an example in $\ell^2(\mathbb Z)$ 
of $T$ not power-bounded satisfying (\ref{CSB}). Zwart \cite{Zw} gave a simpler proof 
of power-boundedness, in any Banach space, when $T$ and $T^*$ both satisfy (\ref{CSB}).
In (a) $\Leftrightarrow$ (d) of \cite[Theorem 2.3]{CS}, 
Chen and Shaw extended Zwart's result;
%(obtained with $\alpha=0$ and $p=2$);
however, since for {\it positive} sequences Ces\`aro boundedness and Abel boundedness are 
equivalent (e.g. \cite[1.5-1.7]{Em}), the use of ``Abel square boundedness"  in \cite{CS}  
is not more general.
%that $T$ is power-bounded if (and only if) both $T$ and $T^*$ are Abel square bounded, 
%thus extending Zwart's result.  

Since (\ref{CSB}) implies $\|T^n\|= O(\sqrt n)$, Theorem 2.1 of \cite{BBMP}, with 
$\frac12 < \alpha< \frac1p$, gives  examples (in $\ell^p$, $1<p<2$) of absolutely 
Ces\`aro bounded operators which are not  Ces\`aro square bounded.
\smallskip

\subsection{Overview} \label{overview}

We briefly describe the main results in the paper. We are of course interested in operators
which are not power-bounded; either  the Kreiss condition or Ces\`aro boundedness imply that
(in the complex case) the spectral radius is at most 1. In Section 2 we 
%prove that if both $T$ and $T^*$ are absolutely Ces\`aro bounded, then $T$ is power-bounded, and 
derive  for $T$ invertible some results similar to \cite[Theorem 2.2]{BBMP}.
In Section 3 we define $p$-absolute Ces\`aro boundedness, 
which extends absolute Ces\`aro boundedness ($p=1$) and Ces\`aro square boundedness ($p=2$). 
We prove in this case that $\|T^n\| =O(n^{1/p -\varepsilon})$ for some $\varepsilon$ 
(which depends on $T$), and provide an example. In Section 4 we obtain growth rates of 
$\|T^n\|$ in Hilbert space: if $T$ satisfies the Kreiss condition, $\|T^n\|=O(n/\sqrt{\log  n})$; 
if $T$ is absolutely Ces\`aro bounded, $\|T^n\|=O(n^{1/2 -\varepsilon})$ for some 
$\varepsilon >0$ (which depends on $T$); if $T$ is strongly Kreiss bounded, then 
$\|T^n\|=O((\log n)^\kappa)$ for some $\kappa >0$.
We show that strong Kreiss boundedness and absolute Ces\`aro boundedness are independent
(none implies the other). In Section 5 we study the ergodic properties of Kreiss bounded operators.
We show that in reflexive spaces Kreiss boundedness implies Abel ergodicity and
strong convergence of Ces\`aro means of order $\alpha$ when $\alpha >1$.  
For positive operators on reflexive complex Banach lattices, Kreiss boundedness 
implies mean ergodicity.  In Section 6 we list some problems which arise from our work.

\medskip

\section{Ces\`aro boundedness conditions and power-boundedness}

In this section we study conditions for power-boundedness. If $T$ is absolutely Ces\`aro 
bounded with $K_{ac}=1$, then $n=2$ in (\ref{ACB}) yields that $T$ is a contraction.

Obviously, if $T$ is Ces\`aro bounded on $X$, so is $T^*$ on $X^*$. Since there are
absolutely Ces\`aro bounded operators which are not power-bounded \cite[Theorem 2.1]{BBMP},
the next proposition shows that their duals are not absolutely Ces\`aro bounded.
We proved it independetly of \cite[Theorem 2.2]{BBMP}, where it was first proved, so
we leave its statement for the reader's convenience (but omit our proof, since it  is 
similar to that of \cite[Theorem 2.2]{BBMP}).

\begin{prop} \label{zwart}
Let $T$ be a linear operator on a (real or complex) Banach space $X$. If both $T$ and $T^*$
are absolutely Ces\`aro bounded, then $T$ is power-bounded.
\end{prop}

%\begin{proof}
%We modify Zwart's ideas \cite{Zw}.
 %Fix $x \in X$ and $x^* \in X^*$. For $N \ge 1$ we have
%$$
%(N+1)|\langle x^*,T^Nx\rangle |^{1/2} = \sum_{k=0}^N |\langle T^{*k}x^*,T^{N-k}x \rangle |^{1/2} \le
%\sum_{k=0}^N\|T^{*k}x^*\|^{1/2}\|T^{N-k}x\|^{1/2} \le
%$$
%$$
%\Big(\sum_{k=0}^N\|T^{*k}x^*\|\Big)^{1/2} \Big(\sum_{k=0}^N\|T^kx\|\Big)^{1/2} .
%$$
%Hence, with $\ac(T)$ and $\ac(T^*)$ denoting the corresponding  constants of (\ref{ACB}), we obtain
%$$
%|\langle x^*,T^Nx\rangle | \le 
%%\Big(\frac1{N+1}\sum_{k=0}^N\|T^{*k}x^*\|\Big) \Big(\frac1{N+1}\sum_{k=0}^N\|T^kx\|\Big) \le
%\ac(T^*)\|x^*\| \cdot \ac(T) \|x\|.
%$$
%Since this is for every $x^* \in X^*$ and $x \in X$, we conclude thata
 %$\|T^N\| \le \ac(T)\ac( T^*)$.
%\end{proof}

%{\bf Remark.} A. Bonilla informed us that Proposition \ref{zwart} was proved 
%independently in Theorem 2.2 of the final version of \cite{BBMP}.
\smallskip

The following theorem answers Question 2.2 of \cite{BBMP} (and improves Corollary 2.4 there).

\begin{theo} \label{noACB}
There exists an invertible operator  $T$ on $\ell^2(\mathbb Z)$ satisfying the 
strong Kreiss resolvent condition which is not absolutely Ces\`aro bounded.
\end{theo}
\begin{proof} Assume that every invertible $T$ on $H=\ell^2(\mathbb Z)$  satisfying (\ref{SKR}) 
is absolutely Ces\`aro bounded. Since $R(\lambda, T^*)=R(\lambda,T)^*$ for 
$\lambda \notin \sigma(T)$, $T^*$ satisfies (\ref{SKR}) whenever $T$ does. Thus 
if $T$ is invertible and satisfies (\ref{SKR}), so does $T^*$, and 
our assumption yields that  $T$ and $T^*$ are both absolutely Ces\`aro bounded.
Hence such $T$ is power-bounded by Proposition \ref{zwart}. 
But McCarthy's example \cite{McC} is an invertible operator $T$ on $H$ which 
satisfies the strong Kreiss resolvent condition and is {\it not} power-bounded -- 
a contradiction to our assumption.
\end{proof}

{\bf Remark.} The construction of Proposition \ref{sk-log} yields examples of $T$
on $L^p$ which are strongly Kreiss and absolutely Ces\`aro bounded (see Proposition
\ref{ask-acb}), but not power-bounded. Hence $T^*$ is strongly Kreiss, but by
\cite[Theorem 2.2]{BBMP} (Proposition \ref{zwart}) it is not absolutely Ces\`aro bounded.
\medskip

The ideas of Zwart \cite{Zw}  yield the following result. 
% related to Problem 2.

\begin{theo} \label{invertible}
Let $T$ be an invertible operator on a (real or complex) Banach space $X$. 
Then the following are equivalent:

(i) Both $T$ and $T^{-1}$ are power-bounded ($T$ is then called 
\emph{doubly power-bounded}).

(ii) Both $T$ and $(T^{-1})^*$ are absolutely Ces\`aro bounded. 

(iii) Both $T^*$ and $T^{-1}$ are absolutely Ces\`aro bounded. 
\end{theo}
\begin{proof} Obviously (i) implies both (ii) and (iii).

Assume (ii). Put $S=T^{-1}$. 
Fix $x \in X$ and $x^* \in X^*$. For $N \ge 1$ we have
$$
(N+1)|\langle x^*,T^Nx\rangle |^{1/2} = \sum_{k=0}^N |\langle S^{*k}x^*,T^{N+k}x \rangle |^{1/2} \le
\sum_{k=0}^N\|S^{*k}x^*\|^{1/2}\|T^{N+k}x\|^{1/2} \le
$$
$$
\Big(\sum_{k=0}^N\|S^{*k}x^*\|\Big)^{1/2} \Big(\sum_{k=0}^N\|T^{N+k}x\|\Big)^{1/2} \le
\Big(\sum_{k=0}^N\|S^{*k}x^*\|\Big)^{1/2} \Big(\sum_{j=0}^{2N}\|T^{j}x\|\Big)^{1/2} .
$$
Hence, with $\ac(T)$ and $\ac(S^*)$ denoting the corresponding  constants of (\ref{ACB}), we obtain
$$
|\langle x^*,T^Nx\rangle | \le 
\Big(\frac1{N+1}\sum_{k=0}^N\|S^{*k}x^*\|\Big) \Big(\frac1{N+1}\sum_{j=0}^{2N}\|T^kx\|\Big) \le
2\ac(S^*)\|x^*\| \cdot \ac(T) \|x\|.
$$
Since this is for every $x^* \in X^*$ and $x \in X$, we conclude that $\|T^N\| \le 2\ac(T)\ac(S^*)$.

To show that $T^{-1}$ is power-bounded, we write 
$$
(N+1)|\langle S^{*N}x^*,x\rangle |^{1/2} = \sum_{k=0}^N |\langle S^{*(N+k)}x^*,T^{k}x \rangle |^{1/2} 
\le \sum_{k=0}^N\|S^{*(N+k)}x^*\|^{1/2}\|T^{k}x\|^{1/2} 
$$
and obtain similarly that $ |\langle  S^{*N}x^*,x\rangle | \le 2\ac(S^*)\|x^*\| \cdot \ac(T) \|x\|$
for every $x^* \in X^*$ and $x \in X$, which yields $\|T^{-N}\| \le 2 \ac(S^*) \ac(T)\ $.
\smallskip

The proof that (iii) implies (i) is similar, so we omit it.
\end{proof}

\begin{prop} \label{isometry}
Let $T$ be absolutely Ces\`aro bounded, and assume that for some $c>0$ 
\begin{equation} \label{below}
\limsup_{n\to\infty} \frac1n\sum_{k=1}^n \|T^kx\| \ge c\|x\| \quad \forall x \in X.
\end{equation}
Then $T$ is power-bounded (by $K_{ac}/c$).
\end{prop}
%\begin{proof} 
%For $x \in X$ define $\|| x\|| := LIM\{\frac1n\sum_{k=1}^n \|T^kx\| \}$, where
%$LIM$ denotes a Banach limit \cite[p. 135]{K}. Since $\liminf a_n \le LIM\{a_n\} \le \limsup a_n$ 
%for any bounded sequence $(a_n)$, (\ref{ACB}) and (\ref{below}) yield that 
%$c\|x\| \le \||x\|| \le C\|x\|$ for every $x \in X$. It is easy to verify that $\||x\||$ 
%is a norm on $X$, so it is equivalent to the original norm. 
%Since $T$ is absolutely Ces\`aro bounded, by \cite[Corollary 2.6]{BBMP} 
%$ \frac{\|T^n\|}n \to 0$ as $n\to \infty$. 
%Since Banach limits are linear and preserve ordinary convergence,
%$$
%\||Tx\|| = LIM\big\{\frac1n \sum_{k=1}^n \|T^{k+1}x\| \big\} =
%LIM \big\{\frac1n\sum_{k=1}^n \|T^kx\| -\frac{\|x\|}n+ \frac{\|T^{n+1}x\|}n \big\} =\||x\||.
%$$
%Hence $T$ on $(X,\||\cdot\||)$ is an isometry, and the equivalence of the norms yields that
%$T$ is power-bounded on $(X,\|\cdot\|)$.
%\end{proof}
\begin{proof}
Let $N\in \N$. By \eqref{below} and \eqref{ACB} with $C=K_{ac}$, we have 
\begin{gather*}
c\|T^Nx\|\le \limsup_{n\to \infty}\frac1n\sum_{k=1}^n\|T^{k+N}x\|\le 
\lim_{n\to +\infty}\frac{N+n}{n} \sup_{m\ge 1} \frac1{m+N}\sum_{k=1}^{m+N}\|T^{k}x\|\le 
K_{ac}\|x\|\, .
\end{gather*}
\end{proof}

{\bf Remark.} The assumptions of Proposition \ref{isometry} do not imply invertibility --
all isometries satisfy them.

\begin{theo} \label{unitary}
The following are equivalent for an invertible $T$ on a Banach space $X$:

(i) Both $T$ and $T^{-1}$ are power-bounded.

(ii) $T^{-1}$ is power-bounded and $T$ is absolutely Ces\`aro bounded.

(iii) $T$ is absolutely Ces\`aro bounded and satisfies (\ref{below}).
\end{theo}
\begin{proof}  Clearly (i) iimplies (ii).
%Assume (i): $\sup_{n\in\mathbb Z} \|T^n\| = K < \infty$. Then 
%$\|x\| =\|T^{-n}T^nx\| \le K\|T^nx\|$ for $n\ge 0$, so (\ref{below}) holds with $c =1/K$.

Assume (ii). Let $M:= \sup_k \|T^{-k}\|$. Then $\|x\|=\|T^{-k}T^k x\| \le M\|T^kx\|$,
and averaging yields \eqref{below}.

Assume (iii). Then by Proposition \ref{isometry}, $T$ is power-bounded. Fix $x \in X$
and $n \in\N$. Then for $N>n$ we have
$$
\frac1N \sum_{k=1}^N \|T^kT^{-n}x\| =
\frac1N \sum_{k=1}^n \|T^kT^{-n}x\| + \frac{N-n}N \cdot\frac1{N-n} \sum_{k=n+1}^N \|T^kT^{-n}x\| =
$$
$$
 \frac1N \sum_{k=1}^n \|T^kT^{-n}x\| + \frac{N-n}N \cdot\frac1{N-n} \sum_{j=1}^{N-n} \|T^jx\|
\le \frac1N \sum_{k=1}^n \|T^kT^{-n}x\| + \frac{N-n}N \ac \|x\| \ .
$$
By \eqref{below}, $c \|T^{-n}x\|  \le \limsup_N \frac1N \sum_{k=1}^N \|T^kT^{-n}x\| 
\le \ac\|x\|$. Hence $T^{-1}$ is power-bounded.
\end{proof}

{\bf Remark.} When $X$ is a Hilbert space, the conditions in Theorem \ref{invertible} or in
Theorem \ref{unitary} are equivalent to similarity of $T$ to a unitary operator, by \cite{SN}.
In $L^p$, $2 \ne p \in(1,\infty)$, an invertible doubly power-bounded operator need not be
similar to an invertible isometry \cite{Gil}, \cite{Co}.
\medskip

\section{$p$-absolute Ces\`aro boundedness and growth of powers}

In this section we complement Theorem \ref{noACB}, by exhibiting a Ces\`aro square bounded
operator on $\ell^2(\mathbb N)$ which is not strongly Kreiss bounded. This extends 
\cite[Corollary 2.2]{BBMP}, where the examples are on $\ell^p(\mathbb N)$,  $1<p<2$,
and answers Question 2.1 there.
\smallskip

The following definition includes absolute Ces\`aro boundedness ($p=1$) and Ces\`aro square 
boundedness ($p=2$). It turns out to be a special case of \cite[Definition 6.6]{ABY} (with 
$\alpha=1$).

{\bf Definition.} Let $1 \le p < \infty$. An operator $T$ on a Banach space is called
{\it $p$-absolutely Ces\`aro bounded} if there exists $C>0$ such that
\begin{equation} \label{pACB}
\sup_{n \ge 1} \frac1n \sum_{k=0}^{n-1} \|T^k x\|^p \le C^p\|x\|^p \qquad \forall \  x\in X.
\end{equation}
 We shall denote by $\pac$ the smallest constant for which \eqref{pACB} holds.

\medskip

Clearly, any $p$-absolutely Ces\`aro bounded operator is $r$-absolutely Ces\`aro bounded for
every $1 \le r \le p$, with $K_{r-\rm{ac}} \le \pac$.

The absolutely Ces\`aro bounded operator on $\ell^p(\mathbb N)$ constructed in 
\cite[Theorem 2.1]{BBMP} is shown in the proof to be $p$-absolutely Ces\`aro bounded.

It is easy to see that $p$-absolute Ces\`aro boundedness implies $\|T^n\|= O(n^{1/p})$.
The next proposition improves this trivial upper bound, and yields $\|T^n\|=o(n^{1/p})$
with a "polynomial" rate.

{
\begin{prop} \label{p-abs-norm}
Let $T$ be $p$-absolutely Ces\`aro bounded. Then 
$$\|T^n\|\le C \|T\|n^{(1/p-\varepsilon)}\, ,$$
where $C= \pac 2^{1/p\pac^p}$ and $\varepsilon=1/p \pac^p$.
\end{prop}
\begin{proof} We denote $\pac$ by $K$.
By assumption, for every $n\in \N$ and every $x\in X$ we have 
$\sum_{k=0}^{n-1}\|T^kx\|^p\le K^p n \|x\|^p$. 
%$\sum_{k=0}^{n-1}\|T^kx\|^p\le \pac^p n \|x\|^p$. 
%Notice that 
Since $\|T^{n}x\| \le \|T^{n-k}\|\|T^k x\|$ for every  $0\le k\le n$, we obtain
$$
%\pac^p n \|x\|^p \ge \|T^nx\|^p \sum_{k= 0 }^{n-1}\frac1{\|T^{n-k}\|^p}=  
K^p n \|x\|^p \ge \|T^nx\|^p \sum_{k= 0 }^{n-1}\frac1{\|T^{n-k}\|^p}=  
\|T^nx\|^p \sum_{k=1}^{n}\frac1{\|T^{k}\|^p}\, .
$$
Hence 
$$
\|T^n\|^p \sum_{k=1}^{n}\frac1{\|T^{k}\|^p}\le K^pn\, .
%\|T^n\|^p \sum_{k=1}^{n}\frac1{\|T^{k}\|^p}\le \pac^pn\, .
$$
Then the result follows from the following numerical lemma, 
applied to $u_n=\|T^n\|^p$.
\end{proof}}

\begin{lem} \label{p-abs-lem}
Let $(u_n)_{n\in \N}$ be a sequence with values in $(0,+\infty)$, such that there exists $C\ge 1$,
 such that for every $n\in \N$, 
$$
u_n\le \frac{Cn}{\sum_{k=1}^{n} \frac1{u_k}}\, .
$$
Then, for every $n\in \N$, 
$$
u_{n}\le C2^{1/C}u_1n^{1-1/C}\, .
$$
\end{lem}
%\noindent {\bf Remark.} Under the assumptions of the lemma, we must have $C\ge 1$
\begin{proof} Set $S_N:= \sum_{k=1}^{N} \frac1{u_k}$. Our assumption implies that for every 
integer $N\ge 2$,
$$
\frac{S_{N}-S_{N-1}}{S_N}\ge \frac1{CN}\, .
$$
%Notice that $S_{N+1}=S_N+ \frac{1}{u_N}\le S_N +1\le S_N+\frac1{u_1}\le 2S_N$. Notice also that 
We then have
$$
\int_{S_{N-1}}^{S_{N}}\frac{dx}x\ge \frac{S_{N}-S_{N-1}}{S_{N}}\ge \frac{1}{CN}\ge 
\frac{1}{C}\int_N^{N+1} \frac{dx}x\, .
$$
Summing those inequalities for $N\in [2, n]$, we obtain that 
$$
\ln S_{n}-\ln S_1 \ge \frac{\ln (n+1)-\ln 2}{C}\, .
$$
Hence, $S_n\ge \frac{S_1 (n+1)^{1/C}}{2^{1/C}}\ge  \frac{S_1 n^{1/C}}{2^{1/C}}$, which yields 
$u_n \le \frac{Cn}{S_n} \le C\frac{2^{1/C}}{S_1} n^{1-1/C}$.
\end{proof}

{\bf Remarks.} 1. When $p=1$, the proposition improves Corollary 2.6 of \cite{BBMP},
where it is proved only that $\|T^n\|=o(n)$. 

2. The power of $n$ in Proposition  \ref{p-abs-norm} is best possible. 
Indeed, in \cite{BBMP}, it is proved that for every 
$p\ge 1$ and every $0<\varepsilon <1/p$, there exists a $p$-absolutely Ces\`aro bounded 
operator $T$ on $\ell^p(\N)$, such that $\|T^n\|=(n+1)^{1/p-\varepsilon}$.

3. Every positive Ces\`aro bounded operator $T$ on $L^1$ is absolutely Ces\`aro bounded. 
Hence, we recover Theorem 2 of Kornfeld and Kosek \cite{KK}. Actually, it happens that an 
application of their Corollary 2, with $\alpha_n:=\|T^nx\|^p$, yields a different proof 
of Proposition  \ref{p-abs-norm}.

4. Abadias and Bonilla \cite{AB} extended the definition of absolute Ces\`aro boundedness in 
a different direction. $T$ is defined to be {\it absolutely Ces\`aro-$\alpha$ bounded} if
$\sup_n M_n^{(\alpha)}(\{\|T^n\|\}) < \infty$, where $M_n^{(\alpha)}$ is the Ces\`aro mean of 
order $\alpha$ \cite[Chapter III]{Zy}. It is proved in \cite{AB} that if $T$ is absolutely 
Ces\`aro-$\alpha$ bounded for $0< \alpha \le 1$, then $\|T^n\|=o(n^\alpha)$; for $\alpha=1$ 
our Proposition \ref{p-abs-norm} (with $p=1$) gives a more precise estimate.

Actually, our method of proof of Proposition \ref{p-abs-norm} allows us to prove 
that absolute Ces\`aro-$\alpha$  bounded operators, $0<\alpha <1$, satisfy an estimate 
$\|T^n\| = O(n^{\alpha -\varepsilon})$ for some $\varepsilon >0$.

\begin{theo}\label{example}
Let $\delta\in (0,1)$ and define the measure $\nu:=\sum_{j\in \N}\frac{\delta_j}{j^\delta}$ 
on $\N$. Let $T$ be the left (backward) shift on $L^p(\N,\nu)$, with $1\le p<\infty$. 
Then, for any fixed $p$, $T$ is $p$-absolutely Ces\`aro bounded, $\|T^n\|_p=(n+1)^{\delta/p}$, 
$T$ is mean ergodic, and $T$ is not strongly Kreiss bounded. 
\end{theo}
\begin{proof}
Fix $p$, and define $V_p\, :\, L^p(\N,\nu)\to \ell^p(\N)$, by 
$V_p(x_j)_{j\in \N}=(\frac{x_j}{j^{\delta/p}})_{j\in \N}$. Note that $V_p$ is an invertible
isometry. The operator $S:=V_pTV_{p}^{-1}$ is nothing but the operator considered in 
Theorem 2.1 of \cite{BBMP}, with $\alpha =\delta/p$ (and $\varepsilon =1 -\delta$). 
Hence, $T=V_p^{-1}SV_p$ is $p$-absolutely Ces\`aro bounded, and $\|T^n\|_p=(n+1)^{\delta/p}$
(we denote by $\|\cdot\|_p$ the norm in $L^p(\N,\nu)$).

Let $\{e_j\}_{j \in \N}$ be the standard basis. Then $T^k e_j=0$ for $k \ge j$, so
$\|\frac1n \sum_{k=1}^n T^k e_j\|_p \to 0$.
Since $T$ is Ces\`aro bounded, we obtain that $\|\frac1n\sum_{k=1}^n T^k x\|_p \to 0$
for every $x \in L^p(\N,\nu)$, so $T$ is mean ergodic.

It remains to prove that $T$ on $L^p(\N,\nu)$ is not strongly Kreiss bounded. 
%done by contradiction. Assume that for some $p\ge 1$, there exists $R>0$ such that for every 
%$z\in \C$ , $\|{\rm e}^{zT}\|_p\le R {\rm e}^{|z|}$. Let $N\in \N$.  
%Let $x=(x_n)_{n\in \N}$ be defined as follows: 
%By the above remark we are only left to proving that $T$ is not strongly Kreiss bounded. 
By contradiction, assume that (\ref{SKR2}) holds: there exists $R>0$ such that for every 
$z\in \C$ , $\|{\rm e}^{zT}\|_p\le R {\rm e}^{|z|}$.  Fix $N\in \N$. Define 
$x=x^{(N)}=(x_n)_{n\in \N}$ as follows: $x_n=1$ if $N+1\le n\le N+2\sqrt N$, and $x_n=0$ 
otherwise.  Then $\|x\|_p^p \le \frac{2\sqrt N}{N^\delta}$. 
\smallskip

Let $z>0$. We have 
$$
\|{\rm e}^{zT}x\|_{p}^p= 
\sum_{k\in \N}\Big(\sum_{n\ge 0}\frac{z^n}{n!}x_{n+k}\Big)^p \frac1{k^\delta} \ge 
\sum_{1\le k\le\sqrt N} \Big(\sum_{N\le n\le N+\sqrt N} \frac{z^n}{n!}\Big)^p \frac1{k^\delta}\, .
$$ 
Taking $z=N$ and using Lemma \ref{stirling} below (with $d=1$), we infer that 

$$
R{\rm e}^N \big( \frac{2\sqrt N}{N^\delta}\big)^{1/p}\ge  
\|{\rm e}^{NT}x\|_{p}\ge 
C \sqrt N \frac{{\rm e}^{N}}{\sqrt N}\Big(\sum_{1\le k\le \sqrt N} \frac1{k^\delta}\Big)^{1/p}\ge 
\tilde C {\rm e}^N N^{(1-\delta)/2p}\, .
$$
Hence, $N^{\delta/2p}\le R2^{1/p}/\tilde C$ which yields a contradiction when $N\to +\infty$.
\end{proof}

\begin{lem}\label{stirling}
There exists $C>0$ such that for every $d>0$, for every $N\in \N$ and every integer  
$1-N \le K \in [-d\sqrt N , d \sqrt N]$,  \quad
$\frac{N^{N+K}}{(N+K)!}\ge C \frac{{\rm e}^{-d^2}}{\sqrt{d+1}} \frac{{\rm e}^{N}}{\sqrt N}$.
\end{lem}
\begin{proof}
By Stirling's formula there exists $R>0$, such that 
$M!\le R \big(\frac{M}{\rm e}\big)^M \sqrt{2M\pi}$ for every $M\in \N$.  
Hence, for every $N\in \N$ and every integer  $-d\sqrt N\le K\le d\sqrt N$, 
\begin{gather*}
(N+K)!    \le R\big(\frac{N}{\rm e}\big)^{N+K} (1+\frac{K}N)^{N+K}\sqrt{2(N+K)\pi } \le \\ 
     R\sqrt{2(d+1)N\pi} \big(\frac{N}{\rm e}\big)^{N+K} {\rm e}^{K+\frac{K^2}N}\le 
\tilde R \sqrt{d+1}\,{\rm e}^{d^2}N^{N+K}{\rm e}^{-N} \sqrt N \, ,
\end{gather*}
and the result follows.
\end{proof}

{\bf Remarks.} 1. When $p=2$, the isometry $V_2$ in the proof of Theorem \ref{example} yields 
that the operator on  $\ell^2(\N)$ constructed in \cite[Theorem 2.1]{BBMP} is 2-absolutely 
Ces\`aro bounded and not strongly Kreiss bounded.

2. For $p=1$, the operator $T$ of Theorem \ref{example} provides another example
of $T$ positive and mean ergodic on $L^1$ with $\|T^n\|=O(n^\delta)$, $\delta$ arbitrarily
close to 1; the first such example was obtained in \cite{KK}.

3. For $p=2$, Theorem \ref{example} yields that the operator $T$ on $\ell^2(\N)$ of 
\cite[Theorem 2.1]{BBMP} is absolutely Ces\`aro bounded, hence uniformly Kreiss bounded, 
but not strongly Kreiss bounded; in \cite[Corollary 2.2]{BBMP} the examples are only for $p<2$.
The examples of $T$ uniformly Kreiss bounded and not strongly Kreiss bounded given in 
\cite[Theorem 5.1]{MSZ} are only in $L^p$, $p\ne 2$, 
\medskip

We now characterize absolute Ces\`aro boundedness by a resolvent type condition. Recall that
when $r(T) \le 1$, for $|\lambda|>1$ we have $R(\lambda,T) =\sum_{n=0}^\infty \lambda^{-n-1}T^n$
with operator norm convergence (e.g. \cite[Lemma 3.2]{LSS}).

{\bf Definition.} We say that an operator $T$ on a complex Banach space $X$ is {\it absolutely 
Kreiss bounded} if there exists $C>0$ such that  
\begin{equation}\label{AK}
\sum_{n\ge 0}\frac{\|T^nx\|}{\lambda^{n+1}}\le\frac{ C\|x\|}{\lambda -1}\qquad 
\forall x\in X,\  \forall \lambda>1\, .
\end{equation}
%We denote by $M_{AK}=M_{AK}(T)$ the smallest constant $C$ for which \eqref{AK} is satisfied.
Clearly, absolute Kreiss boundedness implies uniform Kreiss boundedness.

\begin{prop} An operator $T$ on a complex Banach space is absolutely Ces\`aro bounded if and only 
if it is absolutely Kreiss bounded.
\end{prop}
\begin{proof}
By putting $\alpha=1/\lambda$, (\ref{AK}) becomes 
$\sup_{\alpha \in (0,1)} (1-\alpha)\sum_{n=0}^\infty \alpha^n \|T^n x\| \le C\|x\|$.
Since for {\it positive} sequences Ces\`aro boundedness and Abel boundedness are equivalent 
(e.g. \cite[1.5-1.7]{Em}), we obtain the claimed equivalence.
\end{proof}
%Let us  prove that absolute Ces\`aro boundedness and absolute Kreiss boundedness are equivalent. 
%The fact that $\eqref{AC}\Rightarrow \eqref{AK}$ is well-known (Ces\`aro boundedness implies 
%Abel boundedness). To see the converse, take $\lambda =1 +1/n$ in \eqref{AK} 
%and notice that $(1+1/n)^k\le (1+1/n)^n\le {\rm e}$ for every $0\le k\le n-1$.

{\bf Remark.} The proposition yields that absolute Kreiss boundedness implies uniform Kreiss
boundedness.
\smallskip

{\bf Definition.} An operator $T$ on a (real or complex) Banach space is {\it strongly 
Ces\`aro bounded} (SCB) if there exists $C>0$ such that
\begin{equation} \label{scb}
\sup_{n \in \N} \frac1n \sum_{k=0}^{n-1} |\langle x^*,T^n x\rangle | \le C\|x\|\cdot \|x^*\| \quad
\text { for every } \ x\in X,\quad x^* \in X^*.
\end{equation}
We denote by $K_{scb}$ the smallest $C$ for which (\ref{scb}) holds.

Obviously strong Ces\`aro boundedness implies Ces\`aro boundedness, and
absolute Ces\`aro boundedness implies  SCB. If $T$ is ACB on a reflexive space and not 
power-bounded, then by \cite[Theorem 2.2]{BBMP} (Proposition \ref{zwart}), 
$T^*$ is not ACB, but it is SCB since $T$ is.

\begin{prop} \label{scb1}
$T$ (on a real or complex Banach space $X$) is strongly Ces\`aro bounded if and only if
there exists $C>0$, such that for every sequence of scalars $(\gamma_k)_{k\in \N_0}$ with
$|\gamma_k|=1$ we have
\begin{equation} \label{scb2}
\sup_{n \in \N} \| \frac1n \sum_{k=0}^{n-1} \gamma_k T^k \| \le C.
\end{equation}
\end{prop}
\begin{proof}
If $T$ is strongly Ces\`aro bounded, then for $(\gamma_k)_k$ with $|\gamma_k|=1$ and $n \in \N$ we have
$$
\| \frac1n \sum_{k=0}^{n-1} \gamma_k T^k \| = 
\sup_{\|x\|=1=\|x^*\|} \Big|\frac1n \sum_{k=0}^{n-1}\gamma_k \langle x^*,T^k x \rangle \Big| 
\le K_{scb}.
$$

Assume now that (\ref{scb2}) holds. Fix $x \in X$ and $x^* \in X^*$. Define 
$\gamma_k = \overline{\langle x^*,T^kx \rangle } / |\langle x^*,T^k x \rangle |$ (with the convention 
$\gamma_k=1$ if the terms are zero). Then
$$
 \frac1n \sum_{k=0}^{n-1} |\langle x^*,T^n x\rangle | =
 \frac1n \sum_{k=0}^{n-1} \gamma_k \langle x^*,T^n x\rangle  \le C\|x\|\cdot \|x^*\|.
\vspace{-0.3cm}
$$
\end{proof}

{\bf Remark.} In the complex case, it is enough that (\ref{scb2}) hold for 
$\gamma_k \in \{-1,1\}$.  The proof is similar, taking once 
$\gamma_k = \text{sign } Re \langle x^*,T^k x\rangle $, and then
$\gamma_k = \text{sign } Im \langle x^*,T^k x\rangle $.

\begin{cor} \label{scb*}
$T$ is strongly Ces\`aro bounded if and only if  $T^*$ is.
\end{cor}
\begin{proof} Use the characterization (\ref{scb2}).
\end{proof}

\begin{cor} \label{scb-uk}
If $T$ on a complex Banach space is strongly Ces\`aro bounded, then it is uniformly 
Kreiss bounded.
\end{cor}
\begin{proof} 
For $\gamma \in \mathbb T$, put $\gamma_k=\gamma^k$ in \eqref{scb2},
 and obtain that (\ref{rotated}) holds.
\end{proof}

\section{Growth of powers of operators on Hilbert spaces}

In this section we show that when the operators act on a Hilbert space, we can improve the
estimates on the size of the norms of the powers. We obtain estimates for Kreiss
bounded, absolute Ces\`aro bounded and strongly Kreiss bounded operators.

\begin{theo} \label{kreiss-h}
Let $T$ be a Kreiss bounded operator on a complex Hilbert space $H$. Then 
$\|T^n\|=O(n/\sqrt{\log n})$.
\end{theo}
\begin{proof} By assumption, for every $z>1$, every $\gamma \in \mathbb T$ and $x \in H$,
 we have
$$
\Big\|\sum_{n\ge 0} \frac{T^nx}{(z\gamma)^{n+1}}\Big\|^2 = \|R(z\gamma,T)x\|^2 
\le \frac{K^2\|x\|^2}{(z-1)^2}\, .
$$
\smallskip

Fix $N \in \N$ and take $z=1+1/N$. Integrating the above inequality over $\{|\gamma|=1\}$, 
we obtain 
\begin{gather*}
K^2N^2 \|x\|^2 \ge \sum_{n\ge 0} \frac{\|T^nx\|^2}{(1+\frac1N)^{2(n+1)}}\ge 
(1+\frac1N)^{-2N}\sum_{n=0}^{N-1}\|T^nx\|^2\ge \tilde C \sum_{n=0}^{N-1}\|T^nx\|^2\, .
\end{gather*}

Hence, there exists $C>0$ such that
\begin{equation}\label{bound-KB}
\sum_{n=0}^{N-1}\|T^nx\|^2\le C N^2 \|x\|^2 \qquad \forall N\in \N\, .
\end{equation}

Now, $T^*$ is also Kreiss bounded (with the same constant), hence also $T^*$ satisfies 
\eqref{bound-KB}. Let $0 \le P< Q\le N$ be integers. We have, for every $x,y\in H$,
\begin{gather*}
(Q-P)^2|\langle T^Nx,y\rangle|^2=\Big(\sum_{k=P}^{Q-1} |\langle T^{N-k}x,T^{*k}y\rangle|\Big)^2 \le\\
 \Big(\sum_{k=P}^{Q-1}\|T^{N-k}x\|^2\Big)\Big(\sum_{k=0}^{Q-1}\|T^{*k}y\|^2\Big)\le 
CQ^2\|y\|^2\sum_{k=P}^{Q-1}\|T^{N-k}x\|^2\, .
\end{gather*}

Taking the supremum over $\{\|y\|=1\}$ we infer that 
\begin{equation}\label{claim-BBMP}
\frac{(Q-P)^2}{Q^2} \|T^Nx\|^2 \le C  \sum_{k=P}^{Q-1}\|T^{N-k}x\|^2\, .
\end{equation}
This inequality is just Claim 4 of \cite{BBMP}. 
\smallskip

Let $N\in \N$ and define $L:=\log (N/2)/\log 2$.  It follows from \eqref{claim-BBMP} that 
for every $0\le \ell \le L-1$,
$$
 \sum_{k=N+1-2^{\ell +1}}^{N-2^{\ell}}\|T^kx\|^2\ge \|T^Nx\|^2/4C\, .
$$
Hence, using \eqref{bound-KB},
\begin{gather*}
\|x\|^2 C^2N^2 \ge \sum_{\ell=0}^{L-1} \sum_{k=N+1-2^{\ell +1}}^{N-2^{\ell}}\|T^kx\|^2 \ge 
L\|T^Nx\|^2/4C\, ,
\end{gather*}
and the result follows.
\end{proof}

\begin{cor}
Let $T$ be a Ces\`aro bounded positive operator on a complex Hilbert lattice (which is
necessarily isometrically lattice isomorphic to an  $L^2$ space \cite[p. 128]{MN}).
Then  $T$ is Kreiss bounded, $\|T^n\|=O(n/\sqrt{\log n})$, and $T$ is mean ergodic.
\end{cor}
\begin{proof} $T$ is Kreiss bounded by  Proposition \ref{positive}.
\end{proof}

{\bf Remarks.} 1. Theorem \ref{kreiss-h} improves  \cite[Theorem 2.3]{BBMP}, where 
it is proved that $n^{-1}\|T^n\| \to 0$ when $T$ is {\it uniformly} Kreiss bounded. 
However, the arguments are similar, with some modifications. 

2. Theorem \ref{kreiss-h}  was proved independently by Bonilla and M\"uller \cite{BM}.

3. Nevanlinna \cite[Theorem 0.3 and Corollary 8.2]{Ne1} gave conditions on a Kreiss bounded 
operator in $H$ (that are always satisfied in the finite-dimensional case), 
which imply power-boundedness.

4. Taking $H:= \oplus_{N \ge 1} \mathbb C^N$ and using the construction of 
\cite{STW} on each summand, we can get for any $\varepsilon >0$ a Kreiss bounded operator $T$ 
on a Hilbert space with $ \|T^n\| \ge Cn^{1-\varepsilon}$ for every $n\ge 1$. This operator
is not positive on $H$ (identified with $\ell^2(\N)$).

5. The Example of \cite{KK} yields $T$ on $L^1$ which is absolutely Ces\`aro bounded,
hence uniformly Kreiss bounded, with $\|T^n\| \asymp n^{1-\varepsilon}$, $\varepsilon>0$ small.

6. Bonilla and M\"uller \cite{BM} constructed a uniformly Kreiss bounded $T$ on a Hilbert 
space with $\|T^n\| \asymp n^{1-\varepsilon}$, $\varepsilon>0$ small. Their $T$ is
actually a weighted shift on $\ell^2(\mathbb N)$ with non-negative weights, so by Proposition
\ref{positive} $T$ is even strongly Ces\`aro bounded. In fact, by \cite[Corollary 1]{Sh2},
every weighted shift on $\ell^2(\mathbb N)$ which is uniformly Kreiss bounded is strongly
Ces\`aro bounded. By \cite[Theorem 2.5]{BBMP} 
(or Theorem \ref{acb-csb} below), examples on $H$  with $\|T^n\| \ge c \sqrt{n}$
are not absolutely Ces\`aro bounded.

7. The examples show that the estimate for $\|T^n\|$ in  Theorem \ref{kreiss-h} is 
nearly optimal.

\begin{prop}
Let $T$ be strongly Ces\`aro bounded on a (real or complex) Hilbert space $H$. Then 
$\|T^n\|= O(n/\sqrt{\log n})$.
\end{prop}
\begin{proof}
%It was proved in \cite{BBMP} that a uniformly Kreiss bounded $T$ on a complex Hilbert space
%satsifies $n^{-1}\|T^n\| \to 0$ (it turns out that $\|T^n\| =O(n/\sqrt{n})$, see 
We first observe that when $H$ is a complex Hilbert space, $T$ is uniformly Kreiss bounded
by Corollary \ref{scb-uk}, so the result follows from Theorem \ref{kreiss-h}.
\smallskip

We now prove the real case. Let $H_{\mathbb C}= H\oplus iH$ be the complexification of $H$,
with the norm $\|x+iy\|^2=\|x\|^2+\|y\|^2$, which makes $H_\mathbb C$ a complex Hilbert space
\cite{MDW}. For a  bounded linear operator $S$ on $H$ we define $S_\mathbb C(x+iy) := Sx+iSy$.
Then $S_\mathbb C $ extends $S$ to $H_\mathbb C$, $\ \|S_\mathbb C \| =\|S\|$,  and
$(S^n)_\mathbb C= (S_\mathbb C)^n$.

Let $(\gamma_k)_{k \in \N}$ be a real sequence with $\gamma_k \in \{-1,1\}$. Then 
$$
\|\frac1n \sum_{k=0}^{n-1} \gamma_k (T_\mathbb C)^k \| =
\|\frac1n \sum_{k=0}^{n-1} \gamma_k T^k \|,
$$
so by Proposition \ref{scb1} and the remark following it, also $T_\mathbb C$ is strongly
Ces\`aro bounded. By the result for complex Hilbert spaces, 
$\|T^n\| =\|(T_\mathbb C)^n\| =O(n/\sqrt{\log n})$.
\end{proof}

\begin{theo} \label{acb-csb}
Let $T$ be an absolutely Ces\`aro bounded operator on a Hilbert space $H$. 
Then $T$ is Ces\`aro square bounded, with $\cs\le 8 \ac$. 
Consequently, there exists $\varepsilon\in (0,1/2)$, 
such that $\|T^n\|=O(n^{1/2-\varepsilon})$ with $\varepsilon\le 1/128\ac^2$.
\end{theo}
\begin{proof} The norm estimate will follow from Proposition \ref{p-abs-norm} with $p=2$.

We first  assume that $H$ is complex.

 Let $x\in H$ with $\|x\|=1$. Let $\varepsilon \in (0,1)$.
For $N \in \N$ and $\gamma \in \mathbb T$ define
$$
y_{N,\gamma }:=  \sum_{k=1}^{2^N}\frac{\gamma^k T^k x}{\|T^kx\|+\varepsilon}
$$
and 
$$
u_{N,\gamma}:= \sum_{j=0}^{2^N-1} \gamma^j T ^j y_{N,\gamma}\, . 
$$
Since $T$ is absolutely Ces\`aro bounded, we have $\|u_{N,\gamma}\|\le \ac2^N\|y_{N,\gamma}\|$.

Expanding $\langle y_{N,\gamma},y_{N,\gamma}\rangle$ and using orthogonality, we obtain
$\int_{\mathbb T} \|y_{N,\gamma}\|^2 d\gamma\le  2^N$.  
Consequently, $\int_{\mathbb T} \|u_{N,\gamma}\|^2d\gamma \le \ac^2 2^{3N}$. 

\smallskip

Notice also that 
\begin{equation}\label{u_m}
u_{N,\gamma}=\sum_{j=1}^{2^N}\gamma^j T^j x\sum_{k=1}^j \frac1{\|T^kx\|+\varepsilon}+
\sum_{j=2^N+1}^{2^{N+1}-1}\gamma^j T^j x\sum_{k=j-2^N}
^{2^N} \frac1{\|T^kx\|+\varepsilon}\, .
\end{equation}
\smallskip

Now, expanding $\langle u_{N,\gamma},u_{N,\gamma}\rangle$ and using orthogonality, we obtain
$$
\ac^2 2^{3N} \ge \int_{\mathbb T} \|u_{N,\gamma}\|^2 d\gamma = 
$$
\begin{equation} \label{minorize}
 \sum_{j=1}^{2^N} \|T^jx\|^2\Big(\sum_{k=1}^j \frac1{\|T^kx\|+\varepsilon }\Big)^2+
\sum_{j=2^N+1}^{2^{N+1}-1}
 \|T^jx\|^2\Big(\sum_{k=j-2^N} ^{2^N} \frac1{\|T^kx\|+\varepsilon}\Big)^2
\ge 
\end{equation}
$$
 \sum_{j=2^{N-1}+1}^{2^N} \|T^jx\|^2\Big(\sum_{k=1}^{2^{N-1}} \frac1{\|T^kx\|+\varepsilon}\Big)^2 \, ,
$$
where we minorized the second half of the first sum in \eqref{minorize}.
Notice that (using $\|x\|=1$),
\begin{equation} \label{reciprocals}
2^{N-1}=\sum_{k=1}^{2^N-1}\frac{\sqrt{\|T^kx\|+\varepsilon}}{\sqrt{\|T^kx\|+\varepsilon}}\le 
\Big(\sum_{k=1}^{2^N-1}(\|T^kx\|+\varepsilon)\Big)^{1/2}  \Big(\sum_{k=1}^{2^N-1} \frac1{\|T^kx\|+\varepsilon}\Big)^{1/2}
 \le  
\end{equation}
$$
2^{N/2}\sqrt{\ac+\varepsilon}\Big(\sum_{k=1}^{2^N-1} \frac1{\|T^kx\|+
\varepsilon}\Big)^{1/2} \, .
$$
Hence 
$$
\sum_{k=1}^{2^N-1} \frac1{\|T^kx\|+\varepsilon}\ge \frac{2^{N-2}}{(\ac+\varepsilon)}\, .
$$
Finally we infer that 
$$
\ac^22^{3N} \ge \int_{\mathbb T} \|u_{N,\gamma}\|^2 d\gamma \ge 
\frac{2^{2N-4}}{(\ac +\varepsilon)^2}\sum_{j=2^{N-1}+1}^{2^N}\|T^jx\|^2\, .
$$
Letting $\varepsilon\to 0$ we obtain
$$
\sum_{j=2^{N-1}+1}^{2^N}\|T^j x\|^2\le 2^{N+4}\ac^4 \, .
$$

Let $n\in \N$. Let $N\ge 1$ be such that $2^{N-1}\le n\le 2^{N}-1$. Summing the blocks we get
\begin{equation} \label{final-square}
\sum_{k=0}^{n-1}\|T^kx\|^2\le \sum_{j=0}^{N}2^{j+4}\ac^4 \le 64n\ac^4\, .
\end{equation}
For general $x \ne 0$, replace $x$ in (\ref{final-square}) by $\frac x{\|x\|}$ to obtain the result.
\medskip

When $H$ is a real Hilbert space, we show that on the complexification $H_\mathbb C$ the operator
$T_\mathbb C$ is absolutely Ces\`aro bounded:
$$
\frac1n \sum_{k=0}^{N-1} \|(T_\mathbb C)^k (x+iy) \| \le
\frac1n \sum_{k=0}^{N-1} [\|(T_\mathbb C^kx\| + \|T_\mathbb C^ky \|] \le C(\|x\|+\|y\|) \le
\sqrt{2} C \sqrt{\|x\|^2+\|y\|^2}.
$$
We now apply the result from the complex case: $T_\mathbb C$ is Ces\`aro square bounded, so 
trivially so is $T$.
\end{proof}

\noindent {\bf Remarks.} 1. The Theorem gives a rate in Theorem 2.5 of \cite{BBMP},
 where it is proved that $\|T^n\| =o(n^{1/2})$.

2. The Theorem answers Question 2.3 of \cite{BBMP}.
\medskip

\begin{theo}\label{strong-kreiss}
Let $T$ be a strongly Kreiss bounded operator on a complex Hilbert space $H$. 
Then there exists $\kappa>0$ (which depends on $T$), such that $\|T^n\|=O((\log n)^{\kappa})$.
\end{theo}
\noindent {\bf Remark.} The theorem gives another proof that for $p=2$ the operator in 
Theorem \ref{example} is not strongly Kreiss bounded. In the course of the proof, we 
obtain that $\sum_{n=0}^{N-1} \|T^nx\|^2\le C_\kappa N(\log (N+1))^\kappa\|x\|^2$ for 
every $x\in H$, with the $\kappa>0$ appearing in the theorem.
\medskip

\noindent{\it Proof.} Let $z>0$. By (\ref{SKR2}), for every $x\in H$ and every complex number 
$\gamma$ with $|\gamma|=1$, we have $\|{\rm e}^{z\gamma T}x\|^2 \le M^2 {\rm e}^{2z}\|x\|^2$. 
Expanding the left hand side as a double series of scalar products, integrating over 
$\mathbb T$ with respect to $\gamma$ and using orthogonality of $(\gamma^n)_n$, we infer that 
$$ 
\sum_{n\in \N}\frac{z^{2n}}{(n!)^2}\|T^n x\|^2\le M^2 {\rm e}^{2z} \|x\|^2\, .
$$
Let $N\in \N$ and $d>0$. Putting $z=N$ and applying Lemma \ref{stirling}, we obtain that
$$
C^2 {\rm e}^{-2d^2}\frac{{\rm e}^{2N}}{N}\sum_{N-d\sqrt N\le n\le N} \|T^nx\|^2\le 
M^2 {\rm e}^{2N} \|x\|^2\, ,
$$
with $C=C(d)$. Hence, 
\begin{equation}\label{strong}
\sum_{N-d\sqrt N\le n\le N} \|T^nx\|^2\le \frac{M^2 {\rm e}^{2d^2}}{C^2} N \|x\|^2\, .
\end{equation}

We will complete the proof after the following three lemmas.

\begin{lem}\label{bound}
Let $T$ be an operator on a Banach space $X$.
Let $\alpha\in (0,1]$  and assume that there exists $C\ge 1$ such that for every integer 
$N \ge 4$ and every $x\in X$, $\sum_{N-2\sqrt N\le n\le N} \|T^nx\|^2\le C N^\alpha \|x\|^2$. 
Then for every integer $N\in \N$, 
$\sum_{k=1}^N\|T^kx\|^2\le 8C (1+\|T\|^2)N^{(2\alpha +1)/2}\|x\|^2$.
\end{lem}
\begin{proof} For every integer $M\ge 2$
$$
\sum_{n=1}^{M^2}\|T^nx\|^2=\|Tx\|^2 +  \sum_{k=1}^{M-1} \sum_{n=k^2+1}^{(k+1)^2} \|T^nx\|^2\, .
$$

Let $M\in \N$. Applying our assumption with $N\in \{2^2,\ldots , M^2\}$, we obtain that 
for every $x\in H$,
$$
\sum_{n=1}^{M^2}\|T^nx\|^2 \le \|T\|^2\|x\|^2+C \sum_{k=2}^M k^{2\alpha} \|x\|^2\le  
C(1+\|T\|^2)M^{2\alpha +1}\|x\|^2\, .
$$
Hence, for every $N\in \N$ and every $x\in X$,
\begin{equation*}
\sum_{n=1}^N  \|T^nx\|^2\le C(1+\|T\|^2)([\sqrt {N}]+1)^{2\alpha +1}\|x\|^2 \le 
C (1+\|T\|^2) 2^{2\alpha+1}N^{(2\alpha +1)/2}\|x\|^2\, .
\end{equation*}
Note that by the assumption, $C$ depends on $\alpha$ (and $T$).
\end{proof}

\begin{lem}\label{bound-2}
Let $T$ be strongly Kreiss bounded on $H$. Let $\beta\in(1,3/2]$. Assume that there exists $C>0$ 
such that for every $N\in \N$ and every $x\in H$, 
$\sum_{ n=1}^N \|T^nx\|^2\le CN^\beta \|x\|^2$. Then, there exists $D>0$ (independent of 
C and $\beta$) such that for every integer $N\ge 4$ and every $x\in H$, 
$\sum_{N-2\sqrt N\le n\le N}\|T^kx\|^2\le CD N^{\beta /2}\|x\|^2$, and 
%In particular,  
\begin{equation} \label{Bound-2}
\sum_{ n=1}^N \|T^nx\|^2\le
 8C D(1+\|T \|^2)N^{(\beta+1)/2} \|x\|^2.
\end{equation}
\end{lem}
\begin{proof} Let $x\in H$ with $\|x\|=1$. Let $N\in \N$ and $\gamma\in \C$ with $|\gamma|=1$. 
Set $y_{N,\gamma}:=\sum_{1\le n\le 4\sqrt N} \gamma^n T^n x$ and 
$w_{n,\gamma}:=\sum_{k\ge 0} \frac{\gamma^k}{k!}N^k T^ky_{N,\gamma}$.
\medskip

\noindent
By our assumptions, $\int_{|\gamma|=1} \|y_{N,\gamma}\|^2 \le 4^\beta CN^{\beta/2}  \|x\|^2$, 
and (\ref{SKR2}) yields $\|w_{N,\gamma}\|\le M {\rm e}^{N}\|y_{N,\gamma}\|$.  Hence, 
$$
\int_{\mathbb T } \|w_{N,\gamma}\|^2d\gamma\le 4^\beta M^2 C{\rm e}^{2N}N^{\beta/2} \|x\|^2 \, .
$$
\medskip

On the other hand,
\begin{gather*}
w_{N,\gamma}= \sum_{k\ge 0}\sum_{k+1 \le n\le k+4\sqrt N} \frac{N^k}{k!}\gamma^nT^n x = \\ 
\sum_{1\le n\le 4\sqrt N}\gamma^n T^n x\sum_{k=0}^n 
\frac{N^k}{k!} +\sum_{n\ge 4\sqrt N+1}\gamma^n T^n x\sum_{n-4\sqrt N\le k\le n}\frac{N^k}{k!}\, .
\end{gather*}
Hence, for $N\ge 38$  we have $N-2\sqrt N\ge 4\sqrt N+1$, so using Lemma \ref{stirling} with $d=2$, 
we infer that there exists a constant $E>0$ independent of $N$, such that
\begin{gather*}
\int_{\mathbb T } \|w_{N,\gamma}\|^2\ge 
\sum_{N-2\sqrt N\le n\le N}\Big(\|T^nx\|^2 \big(\sum_{n-4\sqrt N\le k\le n} \frac{N^k}{k!} \big)^2\Big)
 \ge \\ 
\Big(\sum_{N-2\sqrt N\le n\le N}\|T^nx\|^2\Big)
\Big( (\sum_{N-4\sqrt N\le k\le N-2\sqrt N} \frac{N^k}{k!})^2\Big)\ge 
E{\rm e}^{2N} \sum_{N-2\sqrt N\le n\le N }\|T^nx\|^2\, .
\end{gather*}
This provides the first bound with $D=16M^2/E > 4^\beta M^2/E$, when $N\ge 38$. Taking $D$ 
(which is independent of $C$ and $\beta$) greater if necessary,
 the first bound also holds for $4\le N\le 37$. 

The estimate (\ref{Bound-2}) now follows from Lemma \ref{bound}, noticing that $\beta/2\in (0,1]$.
\end{proof}

\begin{lem}
Let $T$ be strongly Kreiss bounded on $H$.  Then there exists $C>0$ such that for every $N\in \N$, 
every $x\in H$ and every integer $K\ge 0$,  we have 
\begin{equation} \label{to-log}
\sum_{k=1}^N \|T^kx\|^2 \le C (8D)^K(1+\|T\|^2)^K N^{1+2^{-K-1}}\|x\|^2.
\end{equation}
Moreover, there exist $S>0$ and $\kappa>0$ such that for every $N\in \N$ and every $x\in H$, 
$\sum_{k=1}^N \|T^kx\|^2 \le SN(\log (N+1))^\kappa \|x\|^2$.
\end{lem}
\begin{proof} For $K=0$, (\ref{to-log}) follows from \eqref{strong} and Lemma \ref{bound} with 
$\alpha=1$. Then, the estimate follows by an easy induction making use of (\ref{Bound-2}).

Let us prove the second bound. Fix an integer $N\ge {\rm e}^{\rm e}$. Let $K\ge 0$ be the integer 
such that  $2^K\le \frac{\log N}{\log \log N}<2^{K+1}$. Set $c:= \log (8D(1+\|T\|^2))/\log 2$. 
Then 
$$(8D(1+\|T\|^2))^K=(2^K)^c\le \big( \frac{\log N}{\log \log N} \big)^c \le (\log N)^c.$$ 
Also  $N^{2^{-K-1}}=\exp(2^{-K-1}\log N)\le \log N$, and (\ref{to-log}) yields the result with 
$\kappa=c+1$.
\end{proof}
\medskip

Let us finish the proof of Theorem \ref{strong-kreiss}. 
\smallskip

Notice that $T^*$ is also strongly Kreiss bounded, with $K_{sk}(T^*)=K_{sk}(T)$,
%the same constant $M$ than $T$ 
and that $\|T\|=\|T^*\|$. Hence, all our Lemmas apply to $T^*$ with the same constants. 
\medskip

Let $N\ge 2$ and $x, y\in H$, with $\|y\|=1$.  Applying the previous lemma also to 
$T^*$ and $y$ we get
\begin{gather*}
(N-1)^2|\langle T^Nx, y\rangle|^2= \Big(\sum_{k=1}^{N-1} |\langle T^k x, T^{*(N-k)}y\rangle|\Big)^2 
\le  \\
( \sum_{k=1}^{N-1}\|T^kx\|^2)(\sum_{k=1}^{N-1}\|T^{*k}y\|^2)\le S^2N^2(\log (N+1))^{2\kappa}\|x\|^2\, .
\end{gather*}
Now, the result follows by taking the supremum over $\{y\in H\, :\, \|y\|=1\}$.
\hfill $\square$
\medskip

Theorem  \ref{strong-kreiss} improves the bound $o(\sqrt{n}/(\log n)^\kappa)$ 
for any $\kappa>0$, stated (without proof) in \cite[p. 3]{McC}.
Now we prove that the estimate of Theorem \ref{strong-kreiss} is the best possible, 
and indeed $\kappa$ depends on $T$. 

\begin{prop} \label{sk-log}
For every $\kappa>0$ there exists a strongly Kreiss bounded operator $T$ on a complex Hilbert 
space $H$, such that $\|T^n\|= \frac1{(\log 2)^\kappa}(\log (n+2))^\kappa$ for every $n\in \N$.
\end{prop}

\begin{proof} Fix $\kappa >0$ and let $H:=L^2(\N,\nu)$ with 
$\nu= \sum_{n\in \N}\frac{\delta_n}{(\log(n+1))^\kappa}$. Let $T$ be the left shift on $H$. 
Obviously, 
$\|T^ne_{n+1}\|^2 = \frac1{(\log 2)^\kappa}= (\log (n+2)/\log 2 )^\kappa \|e_{n+1}\|^2$.

For $(x_k)_{k\in \N}=x \in H$ we have
$$
\|T^nx\|^2 =\sum_{k=1}^\infty {|x_{k+n}|^2} \frac1{(\log(k+1))^\kappa} \le
\|x\|^2 \cdot \sup_{k\ge 1} \Big(\frac{\log(k+n+1)}{\log(k+1)}\Big)^\kappa = 
\Big(\frac{\log(n+2)}{\log 2} \Big) ^\kappa \|x\|^2,
$$
since $\log (x+n)/\log x$ is decreasing.
Hence, $\|T^n\| =(\log (n+2))^{\kappa/2}/(\log 2)^{\kappa/2}$.
\medskip

%To prove (\ref{SKR2}), we put $\alpha =\kappa/2$, and note that 
%$$
%\|\e^{zT}x\| \le  \sum_{n\ge 0} \frac{|z|^n}{n!}\|T^n x\| \le
%\sum_{n\ge 0} \frac{|z|^n (\log (n+2))^\alpha}{n!(\log 2)^\alpha} \|x\|.
%$$
%For $t>0$ put $p_n= p_n(t):= \sum_{k=0}^n \frac{t^k}{k!}$, denote $L(n)= (\log n)^\alpha$,
%and use Abel summation to obtain
%$$
%\sum_{n=0}^N \frac{t^n (\log (n+2))^\alpha}{n!}=
%\sum_{n=1}^N (p_n -p_{n-1}) L(n+2) +L(2)=
%$$
%$$
%p_N L(N+2) - \sum_{n=1}^{N-1} p_n\big(L(n+3) -L(n+2)\big) - L(3) +L(2)=
%$$
%$$
%(p_N-\e^t) L(N+2) +\e^tL(3) + \sum_{n=1}^{N-1} (\e^t-p_n)(L(n+3) -L(n+2)) - L(3) +L(2) \le
%$$

To prove (\ref{SKR2}), we note that by Cauchy-Schwarz, 
$\|\e^{zT}x\|^2 \le \e^{|z|} \sum_{n\ge 0} \frac{|z|^n}{n!}\|T^nx\|^2 $, 
so it suffices to find $C>0$ such that for every $r>0$ and $x\in H$, 
$\sum_{n\ge 0}\frac{r^n}{n!}\|T^nx\|^2\le C{\rm e}^{r} \|x\|^2$. 
\medskip

Fix $r>0$ and $(x_k)_{k\in \N} =x \in H$. Since 
$\frac{\log(k+1-n)}{\log(k+1)}\ge \frac{\log 2}{\log(n+2)}$ for $0\le n<k$ and
$r^n {\rm e}^{-r}\le n^n{\rm e}^{-n}$ for every $n\in \N$, we obtain
\begin{gather*}
\sum_{n\ge 0}\frac{r^n}{n!}\|T^nx\|^2 = 
\sum_{n\ge 0}\frac{r^n}{n!}\sum_{k\in \N} \frac{|x_{k+n}|^2}{(\log (k+1))^\kappa}
= \sum_{n\ge 0}\frac{r^n}{n!}\sum_{k\ge n+1} \frac{|x_{k}|^2}{(\log (k+1-n))^\kappa} \le \\
\sum_{k\ge 1} \frac{|x_k|^2}{(\log (k^{1/4}+1))^\kappa}
 \sum_{0\le n\le k-k^{1/4}} \frac{r^n}{n!}+ 
{\rm e}^r\sum_{k\ge 1}  \frac{|x_k|^2}{(\log (k+1))^\kappa}\sum_{k-k^{1/4}< n \le k-1} 
\frac{n^n {\rm e}^{-n}(\log (n+2))^\kappa}{n!(\log 2)^\kappa} \ .
\end{gather*}
Using  boundedness of $\log(k+1)/\log(k^{1/4}+1)$ (its limit as $k\to \infty$ is 4) to bound
the first term and Stirling's formula to bound the second term, we obtain
\begin{gather*}
\sum_{n\ge 0}\frac{r^n}{n!}\|T^nx\|^2 \le \\
C_1 \sum_{k\ge 1} \frac{|x_k|^2}{(\log (k+1))^\kappa} \sum_{0\le n\le k-k^{1/4}} \frac{r^n}{n!} + 
{\rm e}^r\sum_{k\ge 1}  \frac{|x_k|^2}{(\log (k+1))^\kappa}\sum_{k-k^{1/4}< n \le k-1} 
C_2\ \frac{(\log(n+2))^\kappa}{\sqrt{2\pi n} (\log 2)^\kappa}\le \\
C_1\ \e^r \|x\|^2 + C_2\ {\rm e}^r\sum_{k\ge 1}  \frac{|x_k|^2}{(\log (k+1))^\kappa} 
\Big(k^{1/4} \frac{(\log(k+1))^\kappa}{(\log 2)^\kappa \sqrt{\pi k} } \Big) \le
\e^r C\|x\|^2\, .
\end{gather*}
\end{proof}

{\bf Remarks.} 1. Such a result was proved by McCarthy \cite{McC} for $\kappa=1$ and 
$X=\ell^\infty(\N)$. In our proof $H=L^2(\N,\nu)$, for some $\nu$ which depends on $\kappa$,
but the proof works equally with $X=L^p(\N,\nu)$ for any $1\le p<\infty$.

2. Actually, we have constructed $T$ such that  for some $C>0$ we have
\begin{equation}\label{ASK}
\sum_{n\ge 0}\frac{r^n \|T^nx\|}{n!}\le C{\rm e}^{r}\|x\| \qquad 
\forall x\in X,\ \forall r\ge 0\, .
\end{equation}
%$\sum_{n\ge 0}\frac{r^n}{n!}\|T^nx\|\le C{\rm e}^{r}\|x\|$, for every $r>0$ and every $x\in H$. 
We call $T$ which satisfies \eqref{ASK} {\it absolutely strongly Kreiss bounded}.
We denote by $K_{ASK}=K_{ASK}(T)$ the smallest constant $C$ for which \eqref{ASK} is satisfied.

Absolute strong Kreiss boundedness implies not only strong Kreiss boundedness, 
but also absolute Ces\`aro boundedness (see below); still, 
$T$ is not power-bounded.  Proposition \ref{sk-log} yields examples of strongly Kreiss bounded 
operators which are not power-bounded, different from McCarthy's \cite{McC}.
%strengthens McCarthy's example, which by Theorem \ref{noACB} is not absolutely Ces\`aro bounded.
\smallskip

\begin{prop} \label{ask-acb}
Absolute strong Kreiss boundedness implies absolute Ces\`aro boundedness.
\end{prop}
\begin{proof} Let $N\in \N$. 
By Lemma \ref{stirling}, with $d=1$, there exists $C>0$ such that 
$$
\frac{C {\rm e}^N}{\sqrt N}\sum_{N-d\sqrt N\le k\le N} \|T^kx\|\le 
\sum_{n\ge 0} \frac{N^n\|T^nx\|}{n!}\le K_{ASK}{\rm e}^N\|x\|\, .
$$
Hence $\sum_{N-d\sqrt N\le k\le N} \|T^kx\|\le C'\sqrt{N} \|x\|$.
Proceeding as in the proof of Lemma \ref{bound},
%(which is valid in any Banach space), 
with $\|T^nx\|^2$ replaced by $\|T^nx\|$ (and taking $\alpha =1/2$), we infer that 
there exists $M>0$, such that for  every $N\in \N$,
%$$
%\sum_{k=1}^{N^2}\|T^kx\|\le MN^2 \|x\|\,,
%$$
%and then that for every $N\in \N$
$$
\sum_{k=1}^{N}\|T^kx\|\le M\cdot N \|x\|\, ,
$$
which is precisely absolute Ces\`aro boundedness.
\end{proof}

The following corollary applies to the operators constructed in Proposition \ref{sk-log}.

\begin{cor}
Let $T$ be an absolutely strongly Kreiss bounded operator on $H$ which is not power-bounded.
Then $T^*$ is strongly Kreiss bounded, but not absolutely strongly Kreiss bounded and not
absolutely Ces\`aro bounded.
\end{cor}
\begin{proof}
$T$ is absolutely Ces\`aro bounded, and $T^*$ is obviously strongly Kreiss bounded, 
since $T$ is. If $T^*$ were absolutely strongly Kreiss bounded, it would be absolutely 
Ces\`aro bounded by Proposition \ref{ask-acb}, so by  \cite[Theorem 2.2]{BBMP} (Proposition 
\ref {zwart})  $T$ would be power-bounded, which is a contradiction.
\end{proof}
\smallskip

\section{Ergodic properties under the Kreiss resolvent condition}

Strikwerda and Wade \cite{SW1}, \cite[Theorem 6.1]{SW} proved that $T$ satisfies the 
Kreiss resolvent condition if and only if there is a constant $C$ such that
\begin{equation} \label{sw}
\sup_n \|M_n^{(2)}(\gamma T)\| \le C \quad \forall |\gamma|=1,
\end{equation}
where $M_n^{(2)}(T):= \frac2{(n+1)(n+2)} \sum_{j=0}^n (n+1-j)T^j$ is the $n$th {\it
Ces\`aro mean of order 2} of $T$.
The example in \cite[p. 352]{SW} shows that the Kreiss resolvent condition does 
not imply Ces\`aro boundedness; however, the space there is not reflexive.
\smallskip

We start by extending the characterization of \cite{SW}.
Let us recall the definition of the Ces\`aro means of order $\alpha$ 
(C-$\alpha$ means) and several of their properties. 
We refer to \cite[Section III.1]{Zy}  for those facts. 
\smallskip

For every $\alpha\in \R$, set $A_0^\alpha=1$, and 
$A_n^\alpha:=\frac{(\alpha +1)\ldots (\alpha+n)}{n!}$ for $n\ge 1$.

Then, $A_n^\alpha=\sum_{k=0}^n A_{n-k}^{\alpha -1}$ and 
$A_n^\alpha \sim C_\alpha n^{\alpha}$ as $n\to \infty$. 
\smallskip

Given an operator $T$ on a Banach space, we define 
$M_n^{(\alpha)}=M_n^{(\alpha)}(T)=\frac1{A_n^\alpha}\sum_{k=0}^{n}A_{n-k}^{\alpha-1}T^k$ and 
$S_n^\alpha=S_n^\alpha(T)=A_n^\alpha M_n^{(\alpha)}=\sum_{k=0}^{n}A_{n-k}^{\alpha-1}T^k$. 
Note that $M_n^{(0)}=T^n$ and $M_n^{(1)}= \frac1{n+1}\sum_{k=0}^n T^k$.
\smallskip

For every complex number $z$ with $|z|<1$ one has 
\begin{equation}\label{alpha-identity}
\sum_{n\ge 0} S_n^\alpha z^n= (1-z)^{-\alpha }\sum_{n\ge 0}  z^n T^n \, .
\end{equation}
Here, $(1-z)^{-\alpha} =\exp (-\alpha \log (1-z)) $ where $\log$ is the principal determination 
of the logarithm. In particular, $|(1-z)^{-\alpha}|=|1-z|^\alpha$.
\medskip

{\bf Definition.} We call $T$ {\it Ces\`aro-$\alpha$ bounded} if there exists $C<\infty$ such that 
$$
\sup_{n\ge 0}\|M_n^{(\alpha)}(T)\|\le C\, .
$$
Equivalently, $T$ is Ces\`aro-$\alpha$ bounded if 
$ \sup_{n\ge 0} (n+1)^{-\alpha}\|S_n^\alpha\| < \infty$.

\begin{prop}
Assume that $T$ is Kreiss bounded. Then, for every  $\alpha >1$, 
$$
\sup_{n\ge 2} n^{-\alpha}\|S_n^\alpha\|\le  \frac{2^{\alpha/2} 4 \kb }{\alpha-1}\, .
$$
In particular, $T$ is Ces\`aro-$\alpha$ bounded when $\alpha> 1$.
\end{prop}
\begin{proof} Let $n\ge 2$. Using orthogonality, and then \eqref{alpha-identity} 
with $z=\gamma(1-1/n)$, we obtain
\begin{gather*}
(1-1/n)^n S_n^\alpha =
\int_{|\gamma|=1}\gamma^{-n} \sum_{m\ge 0}S_m^\alpha \gamma^m(1-1/n)^m\, d\gamma =  \\
\int_{|\gamma|=1}\gamma^{-n} (1-\gamma(1-1/n))^{-\alpha}\sum_{m\ge 0}(\gamma(1-1/n))^m T^m \, d\gamma\, .
\end{gather*}

Using the Kreiss boundedness of $T$,  we infer that 
\begin{equation} \label{Sn-alpha}
(1-1/n)^n \|S_n^\alpha\|\le 
\frac{n\kb}{2\pi} \int_{-\pi}^\pi \frac{d\theta}{|1-{\rm e}^{i\theta} (1-1/n)|^\alpha}\ .
\end{equation}
Now, using basic computations and the fact that $|\sin u|\ge 2|u|/\pi$ for $|u|\le \pi/2$ 
(and $1-1/n\ge 1/2$), we obtain
\begin{gather*}
|1-{\rm e}^{i\theta} (1-1/n)|^2= 2(1-1/n)(1-\cos \theta)+1/n^2= 
4(1-1/n)\sin^2(\theta/2)+1/n^2 \ge \\ 
2\sin(\theta^2/2) + 1/n^2 \ge 2 \theta^2/\pi^2+1/n^2\ge \frac12(|\theta|/\pi+1/n)^2\, .
\end{gather*}
Applying this estimate in \eqref{Sn-alpha} and using $(1-1/n)^n \ge (1-1/2)^2 = 1/4$, we obtain 
$$
\|S_n^\alpha\|/4 \le (1-1/n)^n \|S_n^\alpha\|\le 
\frac{2^{\alpha/2} n \kb}{\pi} \int_0^\pi \frac{d\theta}{(\theta/\pi+1/n)^\alpha}\le 
\frac{2^{\alpha/2}\kb n^{\alpha}}{\alpha-1 }\, .
$$
\end{proof}

{\bf Remark.} For $\alpha=1$, the integral in inequality \eqref{Sn-alpha} yields
the result of \cite[Theorem 6.2]{SW}: 
{\it If $T$ is Kreiss bounded, then $\|M_n(T)\|=O(\log n)$.} This is sharp
\cite[p. 352]{SW}.

\begin{cor} \label{alpha1}
The following are equivalent for $T$ on a complex Banach space:

(i) $T$ is Kreiss bounded.

(ii) For every $\alpha >1$ we have 
\begin{equation} \label{Calpha}
\sup_{\gamma \in \mathbb T}\sup_{n\ge 0} \|M_n^{(\alpha)}(\gamma T)\| <\infty\,.
\end{equation}

(iii) \eqref{Calpha} is satisfied by $\alpha=2$.

(iv) \eqref{Calpha} is satisfied by some $\alpha> 1$.
\end{cor}
%The above proof allows to recover Strikwerda-Wade's bound of the usual Ces\`aro averages under 
%the Kreiss resolvent condition. 
\begin{proof} If $T$ is Kreiss bounded, so is each $\gamma T$, with $\kb(\gamma T)=\kb(T)$.
Hence (ii) follows from the Proposition.

Clearly (ii) implies (iii), and (iii) implies (iv).  By \cite{SW} (iii) implies (i).

The implication that (iv) implies (i) follows for instance from \eqref{alpha-identity} 
(applied again with $\gamma T$), noticing that $(1-z)^\alpha= \sum_{n\ge 0} A_n^\alpha z^n$. 
We skip the details.
\end{proof}

{\bf Remark.} Strikwerda and Wade \cite[p. 95]{SW1} proved that $T$ is Kreiss bounded if 
(and only if) \eqref{Calpha} is satisfied for some {\it integer} $\alpha \ge 2$.
\medskip

%Moreover, we see that the Kreiss resolvent condition is equivalent to the fact that for some $\alpha>1$ 
%$$\sup_{|\lambda|=1} \sup_{n\ge 0}(n+1)^{-\alpha} \|\sum_{k=0}^n \lambda^k A_{n-k}^\alpha T^k \|<\infty\, .$$
%
%One implication follows from the proposition, noticing that if $T$ is Kreiss bounded, so is $\lambda T$ 
%with the same constants. The converse implication is well-known and follows for instance from 
%\eqref{alpha-identity} (applied again with $\lambda T$, noticing that 
%$(1-z)^\alpha= \sum_{n\ge 0} A_n^\alpha z^n$.

\medskip

For $r\in (0,1)$ and $T \in B(X)$ with spectral radius $R(T)\le 1$, the series
$A_r(T):= (1-r) \sum_{n=0}^\infty r^nT^n$ converges in operator norm
(since $\limsup_n (r^n\|T^n\|)^{1/n} =r R(T) <1$). We call $A_r(T)$ the {\it Abel
mean}, and if $\sup_{0<r<1} \|A_r(T)\| < \infty$ we say that $T$ is {\it Abel bounded}.
When $\lim_{r\to 1^-} A_r(T)$ exists  strongly, we call $T$ {\it Abel ergodic}.

\begin{prop} \label{kreiss-abel}
If $T$ is Kreiss bounded, then it is Abel bounded.
\end{prop}
\begin{proof}
The Kreiss resolvent condition implies $R(T) \le 1$. For $0<r<1$, putting $\lambda =1/r$ 
in \eqref{KRC} we obtain
$$
\|\sum_{k=0}^\infty r^kT^k\| = \|\sum_{k=0}^\infty \frac{T^k}{\lambda^k}\| =
\|\lambda R(\lambda,T)\| \le \frac{C\lambda}{\lambda-1} =\frac{C}{1-r}\ ,
$$
which proves Abel boundedness.
\end{proof}

We now review the ergodic properties of Abel means, which are mostly well-known.

\begin{prop} \label{abel0}
Let $T$  be Abel bounded. The following are equivalent for $x \in X$:

(i) $\lim_{r \to 1^-}A_r(T)x =0$.

(ii) $x \in \overline{(I-T)X}$.

(iii) $x^*(x)=0$ for every $x^* \in  X^*$ with $T^*x^*=x^*$.

(iv) $A_{r_j}x \to 0$ weakly for some subsequence $r_j \to 1^-$.
\end{prop}
\begin{proof}
It is easy to show $\lim_{r \to 1^-} \|A_r(T)(I-T)\|=0$, so (ii) implies (i).
(i) implies (iv)  and (iv) easily implies (iii). (iii) implies (ii) by the Hahn-Banach theorem
(see \cite[Lemma 3.3]{LSS}.
\end{proof}

\begin{prop} \label{abel-conv}
Let $T$ be Abel bounded on $X$, and put $F(T):= \{y: Ty=y\}$.  Then:

(i) $Y:=F(T)\oplus \overline{(I-T)X}$ is closed.

(ii) $A_rx $ converges as $r \to 1^-$ if and only if $x \in Y$.
\end{prop}
\begin{proof}
For $y \in F(T)$ we have $A_r(T)y=y$; from Proposition \ref{abel0} we obtain that
$F(T)\cap \overline{(I-T)X} =\{0\}$, and that the convergence holds for $x \in Y$.

Assume $A_{r_j}x \to y$ weakly. Since $\|A_r(T)(I-T)\| \to 0$, we have $Ty=y$, and by
Proposition \ref{abel0} we obtain $x-y \in \overline{(I-T)X}$; hence $x \in Y$.
This proves (ii). Then (ii) implies that $Y$ is the set of convergence of $A_r(T)$;
since $\{A_r(T)\}_r$ is bounded, $Y$ is closed.
\end{proof}

The following result is well-known; see \cite[Proposition 3.4]{LSS}.
\begin{cor}
Let $T$ be Abel bounded. Then $T$ is Abel ergodic if and only if
\begin{equation}\label{dec}
X= F(T)\oplus \overline{(I-T)X}\ .
\end{equation}
When $T$ is Abel ergodic, $Ex:=\lim_{r\to 1^-} A_rx$ is the projection on $F(T)$
corresponding to \eqref{dec}.
\end{cor}

\begin{cor} \label{abel-reflex}
Let $T$ be Abel bounded on a reflexive space. Then $T$ is Abel ergodic.
\end{cor}
\begin{proof}
For $x\in X$ there exists $r_j \to 1^-$ such that $A_{r_j}(T)x $ converges weakly, say to $y$.
As before, we see that $A_r(T)x$ converges strongly.
\end{proof}

\begin{cor} 
Let $T$ be Kreiss bounded on a reflexive space. Then $T$ is Abel ergodic.
\end{cor}

\begin{theo} \label{sw-dec}
Let $T$ on $X$ satisfy the Kreiss resolvent condition (\ref{KRC}). Then for $\alpha>1$,
$M_n^{(\alpha)}(T)x$ converges if and only if $x \in F(T)\oplus  \overline{(I-T)X}$.
\end{theo}
\begin{proof}
Since $T$ is Abel bounded, by Proposition \ref{abel-conv}  $Y:=F(T)\oplus\overline{(I-T)X}$ 
is closed. If $M_n^{(\alpha)}(T)x$ converges, then by \cite[Theorem 7]{Hi} $A_r(T)x$ converges
as $r \to 1^-$, so $x\in Y$, by Proposition \ref{abel-conv}.

Fix $1<\alpha < 2$. Corollary \ref{alpha1} implies that $\{\|M_n^{(\alpha)}(T)\|\}$ is 
bounded. Since $M_n^{(\alpha)}(T)x$ is a weighted average of $\{x, Tx,\dots, T^nx\}$, 
it trivially converges for $x \in F(T)$.  

Since $T$ is Kreiss bounded, we have $\|T^n\| =O(n)$ \cite{LN}.
For our $\alpha$, it then follows from \cite[Proposition 2.4]{ED} that 
$\|M_n^{(\alpha)}(T)(I-T)^2x\| \to 0$ for every $x \in X$. We now prove that 
$\overline{(I-T)^2X} = \overline{(I-T)X}$. 
Put $Z:= \overline{(I-T)X}$, which is $T$-invariant, and denote $S:=T_{|Z}$.
Let $y \in Z$. By Proposition \ref{abel0}, $\|A_r(S)y\| =\|A_r(T)y\| \to 0$
as $r \to 1^-$. Hence $y \in \overline{(I-S)Z}$. For $\delta> 0$ there exists
$z  \in Z$ with $\|y-(I-T)z\| < \delta$; then there exists
$x$ with $\|z-(I-T)x\| < \delta/\|I-T\|$. Hence $\|y -(I-T)^2x\|< 2\delta$.
This proves that $Z \subset \overline{(I-T)^2X}$; the converse inclusion is obvious.
Since $\{M_n^{(\alpha)}\}_n$ is bounded, we conclude that $M_n^{(\alpha)}x$ 
converges for every $x \in Y$.

For $\beta \ge 2$, convergence of $M_n^{(\alpha)}x$ implies that of $M_n^{(\beta)}x$
by \cite[Theorem III.1.21]{Zy}.
%For $\alpha=2$, computation yields
%$$
%M_n^{(2)}(T)(I-T)= \frac2{n+2}\big(I-\frac1{n+1}\sum_{k=1}^{n+1}T^k \big).
%$$ 
%By \cite[Theorem 6.2]{SW}, $\|\frac1{n+1}\sum_{k=1}^{n+1} T^k \| =O(\log(n+2))$, so
%$\|M_n^{(2)}(T)(I-T)\| \to 0$, which yields, using boundedness of $M_n^{(2)}(T)$,
 %that $M_n^{(2)}(T)x$ converges for $x \in F(T)\oplus  \overline{(I-T)X}$.
\end{proof}
%
% (iii) Since Ces\`aro convergence (of any order) implies Abel convergence (to the same limit), 
%we have the "if" part of (iii). If $\lim_{\alpha \to 1^-}A_\alpha(T)x =z$, then
%$$
%z-Tz= \lim_{\alpha \to 1^{-}} \A_\alpha(T) (I-T)x =\lim_{n\to \infty} M_n^{(2)}(T)(I-T)x=0,
%$$
%and the previous argument shows that $x-z \in \overline{(I-T)X}$.

{\bf Remarks.} 1. The proof that $M_n^{(2)}(T)x$ converges when 
$x\in F(T)\oplus \overline{(I-T)X}$ is in \cite[Theorem 3.4(i)]{SZ}.
%We have added the "only if" part, which allows a direct proof of (i).

2. For $T$ as in Theorem \ref{sw-dec}, put $Y:= F(T) \oplus \overline{(I-T)X}$, which is 
closed by Proposition \ref{abel-conv}, and obviously $T$-invariant. 
Let $S$ be the restriction of $T$ to $Y$.
Then $S$ satisfies the Kreiss resolvent condition by  (\ref{sw}). Proposition \ref{abel-conv}
yields that $Y=F(S) \oplus \overline{(I-S)Y}$. 

3. For any given $\alpha > 0$, Ed-Dari \cite{ED} gave necessary and sufficient conditions 
for strong convergence of $M_n^{(\alpha)}$.

\begin{cor} 
Let $T$ be Kreiss bounded on a reflexive space. Then for $\alpha >1$,
$M_n^{(\alpha)}(T)x$ converges for every $x \in X$.
\end{cor}
\begin{proof} $T$ is Abel ergodic, and the decomposition \eqref{dec} holds.
\end{proof}

\begin{cor} \label{skr-ME}
Let $T$ on a Banach space $X$ satisfy the strong Kreiss resolvent condition  (\ref{SKR}). 
Then $M_n(T)x:=\frac1n\sum_{k=0}^{n-1} T^k x$ converges if and only if 
$ x \in F(T)\oplus \overline{(I-T)X}$.
\end{cor}
\begin{proof}
By \cite{McC} $\|T^n\|=O(\sqrt n)$ , and $T$ is Ces\`aro bounded, since
by \cite{GZ} it satisfies the uniform Kreiss resolvent condition; 
hence convergence holds for $x\in F(T) \oplus \overline{(I-T)X}$, which is closed by  
Proposition \ref{abel-conv}.
\end{proof}
\smallskip

El-Fallah and Ransford \cite[Corollary 1.4]{FR} show existence of $T$ Kreiss bounded, with
$\|T^n\|/n \not \to 0$; we show that even strong convergence may fail,
%We now show by example that under the Kreiss resolvent condition, the usual ergodic averages 
and the usual ergodic averages $M_n(T)x:=\frac1n\sum_{k=0}^{n-1} T^kx$ need not converge 
for $x\in F(T)\oplus \overline{(I-T)X}$.
\smallskip

{\bf Example.} {\it $T$ satisfying the Kreiss resolvent condition and $y \in X$ with $T^ny/n
\not\to 0$.}

We look at Shields's example \cite{Sh}: $X$ is the space of functions $f$, analytic in the open 
unit disk with $f'$ in $H^1$, with norm $\|f\|:=\|f\|_\infty + \|f'\|_1$. The operator is
$Tf(z)=zf(z)$. Shields proved that $T$ satisfies the Kreiss resolvent condition, and observed 
that $\|T^n\| = n+1$ (hence $T$ does not satisfy the strong Kreiss condition). In fact, for 
$y(z) \equiv 1$ we have $\|T^ny\|=1+n$. 

Taking $x=(I-T)y$ we see that the Ces\`aro averages of $x$ do not converge to 0, so $T$ 
restricted to $Y:= F(T)\oplus \overline{(I-T)X} = \overline{(I-T)X} $ is not mean ergodic. 
$\|M_n(T)x\| =\frac1n\|y-T^ny\|$ is bounded, since $\|T^n\|= n+1$.
Note that $T$ on all of $X$ is not Ces\`aro bounded \cite{SW}, hence does not satisfy the 
uniform Kreiss resolvent condition.
 \smallskip

\begin{cor} \label{KnotME}
There exists $S$ on a Banach space $Y$ which satisfies the Kreiss resolvent condition, 
$Y=F(S) \oplus \overline{(I-S)Y}$, but $S$ is not mean ergodic.
\end{cor}
\medskip

{\bf Remarks.} 1. By Corollary \ref{abel-reflex}, $X$ in the example is not reflexive: 
$F(T) =\{0\}$, but $\overline{(I-T)X} \ne X$, since it contains only functions $g \in X$ 
with $g(1)=0$. 

2. It can be shown that for $T$ of the example $\sigma(T)$ is the closed unit disk, since
 $(\lambda I-T)X \ne X$ when $|\lambda|\le 1$.

3. When $T$ satisfies the Kreiss condition and $\sigma(T)\cap \mathbb T$ has Lebesgue measure 
zero, we have $\|T^n\|/n \to 0$, by \cite[Theorem 5]{Ne} (see also \cite[Corollary 4.6]{AS}).

\medskip

\begin{prop} \label{positive}
The following are equivalent for a positive operator $T$ on a complex Banach lattice $X$.

(i) $T$ is Ces\`aro bounded.

(ii) $T$ is strongly Ces\`aro bounded.

(iii) $T$ is uniformly Kreiss bounded.

(iv) $T$ is Kreiss bounded.

(v) $T$ is Abel bounded.
\end{prop}
\begin{proof} 
Assume (i). Let $(\gamma_k)_{k\in \N_0}  \subset \TT$. For $x \in X$ positivity yields
$$
\Big\|\ \frac1n\sum_{k=0}^{n-1} \gamma_k T^k x \ \Big\|  =
\Big\|\ |\frac1n\sum_{k=0}^{n-1} \gamma_k T^k x| \ \Big\| \le 
\Big\| \frac1n\sum_{k=0}^{n-1} T^k |x|\ \Big\| \le \|M_n(T)\|\cdot\|x\|,
$$
 which yields (ii), by Propostion \ref{scb1}.

Clearly (ii) implies (iii) and (iii) implies (iv). By Proposition \ref{kreiss-abel}, (iv) implies (v).
(v) implies (i), since for positive operators, Abel boundedness implies Ces\`aro boundedness, 
by \cite[I.5]{Em}.
\end{proof}

\begin{cor}
Let $T$ be a positive operator on a reflexive complex Banach lattice $X$.
If $T$ satisfies the Kreiss resolvent condition, then $T$ is mean ergodic.
\end{cor}
\begin{proof} By Proposition \ref{positive}, $T$ is Ces\`aro bounded. Emilion 
\cite[Theorem 4.2]{Em} proved that a Ces\`aro bounded positive
operator on a reflexive Banach lattice is mean ergodic. 
\end{proof}

{\bf Remarks.} 1. Examples of positive operators on $\ell^p(\mathbb N)$ which are 
(uniformly) Kreiss bounded but are not power-bounded, are given in \cite[Theorem 2.1]{BBMP}.
See also Theorem \ref{example}.

2. A positive strongly Ces\`aro bounded operator need not be absolutely Ces\`aro bounded -
the operator $T$ defined in Theorem \ref{example} is positive absolutely Ces\`aro bounded
and not power-bounded, so its dual is SCB but not ACB. See also \cite[Corollary 2.4]{BBMP}.

%\begin{cor} \label{strong}
%Let $T$ on a reflexive Banach space $X$ satisfy the Kreiss resolvent condition. Then the
%following are equivalent:
%
%(i) $T^n/n \to 0$ in the strong operator topology.
%
%(ii) $\lambda T$ is mean ergodic for every $|\lambda|=1$.
%
%(iii) $T$ is mean ergodic.
%\end{cor}
%\begin{proof}
%The equivalence of (i) and (iii) follows from the decomposition in Theorem \ref{dec}.
%
%Fix $\lambda \in \mathbb T$. Since $T$ satisfies the Kreiss condition, so does $\lambda T$.
%But (i) is equivalent to $(\lambda T)^n/n \to 0$ strongly, so (i) is equivalent to (ii) by
%the above.
%\end{proof}

\medskip

\begin{prop} \label{skr-me}
Let $T$ on a reflexive Banach space $X$ satisfy the strong Kreiss resolvent condition 
(\ref{SKR}).  Then $\gamma T$ is mean ergodic for every $\gamma \in \mathbb T$.   
\end{prop}
\begin{proof} By \cite{GZ} $T$ is uniformly Kreiss bounded, so 
$\{\|\frac1n \sum_{k=0}^{n-1} (\gamma T)^k\|\}_{n\ge 1}$ is bounded for fixed $\gamma$,
and by  \cite{McC} (or \cite{LN}), $ \|(\gamma T)^n\|=\|T^n\| =O(\sqrt{n})$.
\end{proof}

{\bf Remark.} When $X$ is a Hilbert space, the above holds if $T$ satisfies only
(\ref{UKR}), since
it was proved in \cite[Corollary 2.5]{BBMP} that if $T$ on a Hilbert space satisfies the
uniform Kreiss resolvent condition, then it is mean ergodic.
\medskip

Recall that on a reflexive Banach space, $T$ is mean ergodic if and only if it is Ces\`aro
bounded and $T^n/n$ converges to zero strongly.

\begin{lem} 
Let $T$ on a reflexive Banach space satisfy the uniform Kreiss resolvent condition. Then the
following are equivalent:

(i) $T^n/n \to 0$ in the strong operator topology.

(ii) $\gamma T$ is mean ergodic for every $|\gamma|=1$.
 
(iii) $T$ is mean ergodic.

(iv) $\gamma_0 T$ is mean ergodic for some $|\gamma_0|=1$.

%(v) $T^*$ is mean ergodic.

%(i) $\|T^n\|/n \to 0$.
\end{lem}
\begin{proof} 
(i) implies (ii) by (\ref{rotated})  and the decomposition \ref{dec}.
Trivially, (ii)$\Longrightarrow$(iii)$\Longrightarrow$ (iv). 
(iv) implies (i) since we get $\|T^nx\|/n =\|(\gamma_0 T)^n x\|/n \to 0$ by (iv).
\end{proof}

{\bf Remarks.} 
1. Weak mean ergodicity of uniformly Kreiss bounded $T$ on a reflexive space is equivalent to
the weak convergence to zero of  $T^n/n$, in view of the decomposition \ref{dec}.
Since $T^*$ satisfies (\ref{rotated}) if and only if $T$ does, $T$ is weakly mean ergodic 
if and only if $T^*$ is.
 
2. The above lemma is valid also for the weak topology.
%The above corollary shows that the question of (strong) mean ergodicity 
%is equivalent to the seemingly stronger question: 
%{\it If $T$ satisfies (UKR) on a reflexive space, does $\|T^n\|/n \to 0$?}
%(a special case of \cite[Question 3]{MSZ}). 

\begin{cor} \label{ukr-me}
Let $T$ on a reflexive Banach space satisfy the uniform Kreiss resolvent condition. 
If the Lebesgue measure of $\sigma(T) \cap \mathbb T$ is zero, then for every 
$\gamma \in\mathbb T$, the operators $\gamma T$ and $\gamma T^*$ are mean ergodic.
\end{cor}
\begin{proof} $T$ and $T^*$ satisfy all the conditions in the lemma, 
since $\|T^n\|/n \to 0$ by \cite{Ne}.
\end{proof}

{\bf Example.} {\it A uniformly Kreiss bounded not power-bounded $T$ on a reflexive space, 
with $\sigma(T)=\{1\}$.}

Let $V$ be the Volterra operator on $L^p[0,1]$, $1<p<\infty$, and for $r>0$ put $T:=I-rV$.
By \cite{MSZ} $T$ is UKB, and for $p\ne 2$ it is not power-bounded and not strongly
Kreiss bounded. $\sigma(T)=\{1\}$ because $\sigma(V)=\{0\}$.  Hence mean ergodicity of $T$
does not follow from Proposition \ref{skr-me}, but it does from Corollary \ref{ukr-me}.
\medskip

It was shown in \cite[Corollary 2.5]{BBMP} that {\it any $T$ on a Hilbert space which 
satisfies the uniform Kreiss resolvent condition is mean ergodic}. Therefore,
if $T$ on $H$ is mean ergodic and $T^*$ is not (e.g. \cite{De}, \cite[Example 3.1]{TZ}), 
then $T$ is not UKB; by Theorem \ref{kreiss-h}, it is not even Kreiss bounded.

\medskip

\section{Problems}

%1. What is the relationship between Ces\`aro square boundedness (CSB) and the strong Kreiss 
%resolvent (SKR) condition? Absolute Ces\`aro boundedness does not imply (SKR) by 
%\cite[Corollary 2.2]{BBMP}, but since the example there is not (CSB), maybe (CSB) does; 
%both conditions have the same (maximal) growth rate $\|T^n\|=O(\sqrt n)$ \cite{McC}, \cite{LN}. 
%Does (SKR) imply (CSB)? It is not known if it implies even absolute Ces\'aro boundedness (ACB).
1. A question related to McCarthy's example \cite{McC} is whether in every complex Banach space 
there is a strongly Kreiss bounded operator which is not power-bounded.
%\medskip
\smallskip

2. Let $T$ be invertible. If both $T$ and $T^{-1}$ are absolutely Ces\`aro bounded
(or even both are Ces\`aro square bounded), are $T$ and $T^{-1}$ power-bounded? 
In Hilbert space, if in addition one of them is power-bounded, so is the other, since absolute 
Ces\`aro boundedness implies the Kreiss resolvent condition, and we can apply the result of van 
Casteren \cite{VC1}. Hence our problem in Hilbert space is to weaken the power-boundedness 
assumption on one of them, made in \cite{VC1}, and strengthen the Kreiss resolvent condition on 
the other.
%Note that McCarthy's example \cite{McC} $T$ is invertible with both $T$ and $T^{-1}$ satisfying 
%(SKR), without being power-bounded.
\smallskip

3. Is there an example (preferably in a Hilbert space) of $T$ uniformly Kreiss bounded and not 
strongly Ces\`aro bounded? By Corollary \ref{scb-uk}, in any Banach space every SCB operator is UKB.
\smallskip

%3. Can the conditions of the Chen-Shaw power-boundedness result \cite{CS} (which extends Zwart's
%result) be weakened?  For example, if $T$ and $T^*$ are both absolutely Abel bounded, is $T$ 
%power-bounded?  Chen and Shaw assume Abel square boundedness, which is stronger. 
%There might be different answers in Hilbert space and in general Banach spaces.
%\medskip

%4. Find examples of absolutely Abel bounded operators which are not absolutely Ces\`aro bounded,
%and Abel square bounded operators which are not Ces\`aro square bounded, and even not absolutely 
%Ces\`aro bounded.
%\medskip 

4. Let $T$ be uniformly Kreiss bounded on $H$. Is there a rate for $\|T^n\|$ which is better
than the rate obtained for Kreiss bounded operators in Theorem \ref{kreiss-h}? Same question
for $T$ strongly Ces\`aro bounded. There is a gap between the growth rates for Kreiss bounded
operators and for absolutely Ces\`aro bounded ones (in $H$). In view of M\"uller's example 
mentioned in the remarks following
Theorem \ref{kreiss-h}, the question is whether the rate is $\|T^n\|= O(n^{1-\varepsilon})$
(with $\varepsilon>0$ depending on $T$).
\smallskip

5. Does uniform Kreiss boundedness of $T$ on a reflexive space imply mean 
ergodicity? In Hilbert spaces the answer is positive, by \cite[Corollary 2.5]{BBMP}. 
%since (UKR) implies $\|T^n\|=o(n)$, by
 However, for mean ergodicity we need only that $T^n/n \to 0$ {\it strongly}, 
and even this is not known in general reflexive spaces. 
 What if we stregthen the assumption to strong Ces\`aro boundedness of $T$?
%A related question is: if $T$ on a reflexive space is mean ergodic and satisfies (UKR), 
%does it imply that $\|T^n\|/n \to 0$? 
This question is related to Question 3 in \cite{MSZ}: 
Does (UKB) imply a rate of growth of $\|T^n\|$ better than $O(n)$?
For the converse, assume that $T$ on $H$ satisfies $\|T^n\|=o(n)$ and that $\gamma T$ is
mean ergodic for every $|\gamma|=1$; is $T$ uniformly Kreiss bounded?
%(i.e, is the Ces\`aro boundedness of $\lambda T$ uniform in $|\lambda|=1$)?
%\medskip
\smallskip

6. Find $T$ on a reflexive Banach space which satisfies the Kreiss resolvent condition and
is not mean ergodic. Is there an example on a Hilbert space?
Note that the space in Corollary \ref{KnotME} is not reflexive.
\smallskip

7. Let $T$ be Ces\`aro square bounded with $\sigma(T)\cap \mathbb T \subset \{1\}$. Does 
$\|T^n(I-T)\| \to 0$?  If yes, what if $T$ is only (ACB)?  By 
%Katznelson-Tzafriri 
\cite{KT}, the answer is positive for $T$ power-bounded, 
but negative for $T$ only Ces\`aro bounded (Tomilov-Zem\'anek \cite{TZ}, L\'eka \cite{Le}),
or when $T$ satisfies only (\ref{KRC}) (Nevanlinna \cite{Ne}). The question whether (UKB) is 
sufficient was posed by Suciu \cite{Su1}; see also \cite{Ne}, \cite{SZ}, \cite{Su}. 
In the above questions, if $T^n(I-T)$ does not converge to 0 in norm, does it converge strongly
(see \cite[Theorem 3.1(iv)]{SZ})?

\bigskip

{\it Note added in proof.} Loris Arnold has noted that Example 1 in \cite{VC2} gives a negative
answer to Problem 2 above.
\bigskip

{\bf Acknowledgements.} The last author is grateful for the hospitality of the universities of 
Leipzig and of Brest, where parts of this research were carried out.

The second author's interest in the subject was motivated by V. M\"uller's talk at the 2019 
workshop of  the internet seminar on ergodic theorems, co-organized  in Wuppertal 
by the third author.

\end{document}